\documentclass[11pt]{article}
\usepackage[T1]{fontenc}
\usepackage{bbm}
\usepackage{bm}
\usepackage[margin=1in]{geometry}
\usepackage{amsfonts}
\usepackage{amsmath}
\usepackage{amsthm}
\usepackage{amssymb}
\usepackage{enumitem}
\usepackage{mathtools}
\usepackage{tikz}
\definecolor{red}{HTML}{F44336}
\definecolor{green}{HTML}{4CAF50}
\definecolor{yellow}{HTML}{FFEE58}
\definecolor{blue}{HTML}{0D47A1}
\usepackage[square,numbers,sort]{natbib}
\usepackage[colorlinks,linkcolor=blue,citecolor=blue]{hyperref}

\allowdisplaybreaks[1]
\numberwithin{equation}{section}

\theoremstyle{plain}
\newtheorem{theorem}{Theorem}[section]
\newtheorem{lemma}[theorem]{Lemma}
\newtheorem{proposition}[theorem]{Proposition}

\newtheorem{corollary}[theorem]{Corollary}

\theoremstyle{definition}
\newtheorem{definition}[theorem]{Definition}

\theoremstyle{remark}
\newtheorem{remark}[theorem]{Remark}

\newlist{thenum}{enumerate}{1}
\setlist[thenum, 1]{label=(\alph*), ref=\thetheorem(\alph*)}

\let\P\relax
\DeclareMathOperator{\P}{\mathbb{P}}
\DeclareMathOperator{\E}{\mathbb{E}}

\DeclareMathOperator{\Var}{\mathbb{V}ar}

\newcommand*{\0}{\bm{0}}

\newcommand*{\inner}[2]{\langle #1, #2\rangle}
\providecommand{\abs}[1]{\lvert#1\rvert}
\providecommand{\norm}[1]{\lVert#1\rVert}

\newcommand*{\ind}[1]{\operatorname{\mathbbm{1}}\{#1\}}
\DeclareMathOperator{\bb1}{\mathbbm{1}}
\newcommand*{\tv}[2]{\operatorname{TV}(#1, #2)}
\newcommand*{\kl}[2]{\operatorname{KL}(#1\parallel#2)}
\newcommand*{\cs}[2]{\operatorname{\chi^2}(#1\parallel#2)}
\newcommand*{\Df}[2]{\operatorname{D}_f(#1\parallel#2)}
\DeclareMathOperator{\Ent}{Ent}

\newcommand*{\RR}{\mathbb{R}}
\renewcommand*{\SS}{\mathbb{S}}

\newcommand*{\cD}{\mathcal{D}}
\newcommand*{\cF}{\mathcal{F}}
\newcommand*{\cG}{\mathcal{G}}
\newcommand*{\cN}{\mathcal{N}}
\newcommand*{\cP}{\mathcal{P}}
\newcommand*{\cQ}{\mathcal{Q}}

\newcommand*{\cX}{\mathcal{X}}

\renewcommand*{\a}{\bm{a}}
\renewcommand*{\b}{\bm{b}}
\newcommand*{\e}{\bm{e}}

\newcommand*{\x}{\bm{x}}
\newcommand*{\y}{\bm{y}}
\newcommand*{\z}{\bm{z}}

\newcommand*{\A}{\bm{A}}
\newcommand*{\B}{\bm{B}}

\newcommand*{\I}{\bm{I}}

\newcommand*{\X}{\bm{X}}

\newcommand*{\bSigma}{\bm{\varSigma}}

\newcommand*{\brsigma}{\bar{\sigma}}

\newcommand*{\ER}{Erd\H{o}s--R\'enyi}

\begin{document}
\title{A probabilistic view of latent space graphs
and phase transitions}
\author{Suqi Liu\thanks{Princeton University, Princeton, NJ 08544, USA;
Email: \texttt{\{suqil, mracz\}@princeton.edu}.
Research supported in part by NSF grant DMS-1811724.
}
\and
Mikl\'os Z.\ R\'acz\footnotemark[1]}
\date{October 29, 2021}

\maketitle
\begin{abstract}
We study random graphs with latent geometric structure,
where the probability of each edge depends on the underlying random positions
corresponding to the two endpoints.
We focus on the setting where this conditional probability is a general
monotone increasing function of the inner product of two vectors;
such a function can naturally be viewed as
the cumulative distribution function of some independent random variable.
We consider a one-parameter family of random graphs,
characterized by the variance of this random variable,
that smoothly interpolates between a random dot product graph
and an Erd\H{o}s--R\'enyi random graph.
We prove phase transitions of detecting geometry in these graphs,
in terms of the dimension of the underlying geometric space
and the variance parameter of the conditional probability.
When the dimension is high or the variance is large,
the graph is similar to an Erd\H{o}s--R\'enyi graph with the same
edge density that does not possess geometry;
in other parameter regimes, there is a computationally
efficient signed triangle statistic that distinguishes them.
The proofs make use of information-theoretic inequalities
and concentration of measure phenomena.
\end{abstract}

\section{Introduction} \label{se:intro}
Random graphs defined over some latent geometric space are widely used to model
a large variety of real-life networks.
Random dot product graphs are a natural family
of random graphs with simple latent geometry
where the probability of connection depends on the inner product of two
vectors in a Euclidean space~\cite{young2007random,athreya2018statistical}.
For a graph on a set of vertices $V = [n] \coloneqq \{1,2,\ldots,n\}$,
we write $i \sim j$ if vertices $i$ and $j$ are connected by an undirected edge.
Let $\x_1, \ldots, \x_n \in \RR^d$ be independent identically distributed
random vectors.
In the general setting of random dot product graphs,
conditioned on the latent vectors $\x_i$ and $\x_j$,
the event $i \sim j$  happens with probability $\sigma(\inner{\x_i}{\x_j})$,
independently of everything else,
where $\inner{\cdot}{\cdot}$ denotes the inner product and $\sigma: \RR \to [0, 1]$ is usually a monotone increasing function called the
\emph{connection function}.
That is,
\begin{equation*}
\P(i \sim j \mid \x_1, \ldots, \x_n) = \sigma(\inner{\x_i}{\x_j}).
\end{equation*}
The distribution of the random graph is then specified by
\begin{equation*}
\P(G) = \E\biggl[\prod_{i<j}\sigma(\inner{\x_i}{\x_j})^{a_{i,j}}
(1-\sigma(\inner{\x_i}{\x_j}))^{1-a_{i,j}}\biggr],
\end{equation*}
where $\A = [a_{i,j}]$ is the adjacency matrix of the graph $G$.

The properties of the connection function $\sigma$
play a key role in the presence of geometry in the random graph.
When $\sigma$ is a threshold function,
that is, $\sigma(x) = \ind{x \ge t}$,
then such a graph is known simply as a \emph{random dot product graph}
\cite{liu2021phase}.
At the other extreme, if $\sigma$ is a constant,
we have an \ER{} random graph, with no geometry.
In between,
intuitively, when the connection function is ``steep'',
the edges are related through the latent geometric space;
when the connection function is ``flat'',
the edges become less dependent of each other.
The flatness of the connection function can also be understood as the level
of noise in a geometric graph:
A connection function that is close to a constant implies large noise.
The trade-off between noise and dimensionality in detecting geometry in
random graphs was first studied by the authors \cite{liu2021phase},
where a particular one-parameter family of connection functions,
consisting of step functions, were studied.
An immediate question is how does this trade-off generalize to other connection
functions (in particular, smooth ones) that are widely used in practice?
We attempt to answer this question by studying a natural one-parameter family of
smooth connection functions that interpolates between the two extremes.

\subsection{A probabilistic view of the connection function}
We focus on the case when the latent positions are independent
standard normal random vectors, that is,
$\x_i \sim \cN(\0, \I_d)$ for $1 \le i \le n$.
We consider a broad class of connection functions bearing a probabilistic view.
Observing that the connection function is usually monotone increasing between
$0$ and $1$,
it is natural to view it as the cumulative distribution function (CDF) of some
random variable.
Keeping this setup in mind, we seek a parametrization that characterizes
the ``flatness'' of the function and interpolates between the geometric graph
and the \ER{} model.
A natural parameter is the variance of this random variable.
If the variance is large, the CDF changes slowly.
On the other hand, if the variance goes to zero, the random variable
converges to a constant and the CDF becomes a threshold function.
We formulate the idea as follows.

Suppose $\cD(0,1)$ is an arbitrary zero mean and unit variance distribution
with the CDF denoted by $F: \RR \to [0,1]$.
Suppose the probability measure is absolutely continuous with respect to
Lebesgue measure and let $f$ be the probability density function (PDF).
We assume that
\begin{equation} \label{eq:A0}
f(x) > 0, \ \ \forall x \in \RR. \tag{A0}
\end{equation}
We also assume that $f$ is continuously differentiable and
the derivative $f'$ is bounded:
\begin{equation} \label{eq:A1}
\alpha \coloneqq \sup_x \abs{f'(x)} < \infty. \tag{A1}
\end{equation}
For technical reasons, we further assume that the second order derivative
$f''$ exists and for any fixed Gaussian random variable $X$,
\begin{equation} \label{eq:A2}
\E[\abs{f''(X)}] < \infty. \tag{A2}
\end{equation}

We then create a one-parameter family of connection functions using $F$.
The process can be viewed as an affine transformation of a random variable
following $\cD(0,1)$.
The ``flatness'' of the connection function is parametrized by $r$,
which measures the deviation of the random variable from a constant.
However, since the inner product of
two $d$-dimensional standard normal random vectors
has a variance of $d$, we account for it by setting the variance to $r^2d$.
We also need to match the marginal probability of an edge in the graph,
which is done by choosing an appropriate mean.
Then, we can view the connection function as the CDF of the distribution
$\cD(\mu_{p,d,r}, r^2d)$:
\begin{equation} \label{eq:sigma}
\sigma(x) := F\biggl(\frac{x-\mu_{p,d,r}}{r\sqrt{d}}\biggr),
\end{equation}
where $\mu_{p,d,r}$ is determined by setting the
edge density in the graph to be equal to~$p$:
\begin{equation*}
\P(i \sim j) = \E[\sigma(\inner{\x_i}{\x_j})]
= \E\biggl[F\biggl(\frac{\inner{\x_i}{\x_j}-\mu_{p,d,r}}{r\sqrt{d}}\biggr)
\biggr] = p.
\end{equation*}
We denote the random graph constructed in this way
by $\cG(n,p,d,r)$.

The edge generating process for $\cG(n,p,d,r)$ can also be viewed
as follows.
For each pair of vertices $i$ and $j$, we draw an independent random variable
$z_{i,j} \sim \cD(\mu_{p,d,r}, r^2d)$.
Then, conditioned on the latent positions $\x_i,\x_j$, and the random variable
$z_{i,j}$,
vertices $i$ and $j$ are connected if and only if
$\inner{\x_i}{\x_j} \ge z_{i,j}$.
In other words, for a graph with adjacency matrix $\A$,
\begin{equation*}
a_{i,j} = \ind{i \sim j} = \ind{\inner{\x_i}{\x_j} \ge z_{i,j}}.
\end{equation*}
In contrast, in the conventional definition of a random dot product graph,
the connection between $i$ and $j$ is determined by comparing the inner product
with a threshold that is a constant~\cite{liu2021phase}.

Canonical examples of the connection function, which satisfy the assumptions above, include:
\begin{itemize}
\item logistic (here $\alpha = \pi^2/(18\sqrt{3})$):
\begin{equation*}
F(x) = \frac{1}{1+\exp(-\pi x/\sqrt{3})},
\end{equation*}
\item Gaussian (here $\alpha = 1/\sqrt{2e\pi}$):
\begin{equation*}
F(x) = \Phi(x)
= \frac{1}{\sqrt{2\pi}}\int_{-\infty}^x \exp\biggl(-\frac{y^2}{2}\biggr)\,dy.
\end{equation*}
\end{itemize}
The assumptions are also satisfied broadly by many more CDFs.

For a random graph with latent geometric structure,
the only observables are the edges,
while the underlying random vectors are invisible.
It is natural to ask the question that for a given graph,
whether geometry is reflected in the combinatorial structure.
A canonical random graph model that does not possess geometry
is the \ER{} graph $\cG(n,p)$,
in which the edges are generated independently with probability $p$.
Therefore, in this work, we focus on
how $\cG(n,p,d,r)$ compares to $\cG(n,p)$.
The detection of geometry can then be understood through bounds on the total
variation distance between them which will be defined formally
in Section~\ref{se:pre}.
Roughly speaking, when the total variation distance is close to $0$,
there is no algorithm that can tell the difference.

\subsection{Main result}
Our main result is summarized as the following theorem.
\begin{theorem} \label{th:main}
Let $\cG(n,p,d,r)$ be defined above
with $p$ fixed in $(0,1)$ and $r \ge 1$.
\begin{thenum}
\item \label{th:imp}
Assume \eqref{eq:A1}.
If
\begin{equation*}
\frac{n^3}{r^4 d} \to 0,
\end{equation*}
then $\tv{\cG(n,p,d,r)}{\cG(n,p)} \to 0$.
\item \label{th:pos}
Additionally, assume \eqref{eq:A0} and \eqref{eq:A2}.
Suppose that $d/\log^2 d \gg r^6$ or $r/\log^2 r \gg d^{1/6}$.
If
\begin{equation*}
\frac{n^3}{r^6 d} \to \infty,
\end{equation*}
then $\tv{\cG(n,p,d,r)}{\cG(n,p)} \to 1$.
\end{thenum}
\end{theorem}

Theorem~\ref{th:main} can be displayed graphically by a phase diagram
in the space of the parameter $r$ and the dimension $d$,
as shown in Figure~\ref{fg:diag}.

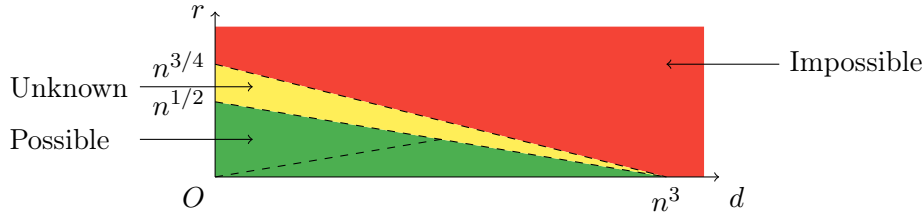
\begin{figure}[t]
\begin{center}
\begin{tikzpicture}[scale=2]
\draw[fill=green,draw=none] (0, 0) -- (0,0.5) -- (3,0);
\draw[fill=red,draw=none] (0,0.75) -- (0,1) -- (3.25,1) -- (3.25,0) -- (3,0);
\draw[fill=yellow,draw=none] (0,0.5) -- (0,0.75) -- (3,0);
\draw[->] node[below left] {$O$} (0,0) -- (3.35,0) node[below right] {$d$};
\draw[->] (0,0) -- (0,1.1) node[left] {$r$};
\draw[dashed] (0,0.5) node[left] {$n^{1/2}$} -- (3,0) node[below] {$n^3$};
\draw[dashed] (0,0.75) node[left] {$n^{3/4}$} -- (3,0);
\draw[dashed] (0,0) -- (1.5,0.25);
\node[left,text width=1.6cm](a) at (-0.5,0.25) {Possible};
\node[left,text width=1.6cm](c) at (-0.5,0.6) {Unknown};
\node[right,text width=1.8cm](b) at (3.75,0.75) {Impossible};
\draw[->] (a) -- (0.25,0.25);
\draw[->] (c) -- (0.25,0.6);
\draw[->] (b) -- (3,0.75);
\end{tikzpicture}
\end{center}
\caption{Phase diagram for detecting geometry in $\cG(n,p,d,r)$.}
\label{fg:diag}
\end{figure}

\begin{remark}
There is an intermediate regime that is not covered by the theorem.
We believe the phase transitions happen at a certain power of $r$.
We conjecture that Theorem~\ref{th:pos} gives the actual threshold.
In the proof of Theorem~\ref{th:imp},
a factor of $r^2$ might have been lost when applying Jensen's inequality
to the KL divergence (see Section~\ref{se:imp}).
\end{remark}

\begin{remark}\label{rem:soft}
The noisy high-dimensional random geometric graphs studied in
\cite{liu2021phase} do not fall into the family of graphs considered
in this work
since the connection functions there are not continuously differentiable,
which is crucial in the proofs.
Nevertheless, we can draw some comparisons between them.
We see that the signed triangle statistic gives the same detection boundary
in both models.
This could be understood as despite the discrepancies, the parameters are both a
linear measurement of the deviation from a constant.
\end{remark}

\begin{remark}
The theorem is stated in the case $r \ge 1$,
which is the regime where there is an interplay between $r$ and $d$.
When $r = o(1)$, by the data processing inequality,
the total variation distance between $\cG(n,p,d,r)$ and $\cG(n,p)$ is upper bounded by
the Wishart to GOE transition,
which gives the condition $d \gg n^3$ for the impossibility of detecting
geometry~\cite{liu2021phase}.
At the same time, by the calculation in \cite[Section~5]{liu2021phase},
the detection power for the signed triangle statistic when $r \to 0$
(the connection function becomes an indicator function) is $d \ll n^3$.
Combining these, we have that the phase transition is at $d \asymp n^3$
regardless of the rate of $r$.
\end{remark}

\subsection{Related work}
The study of random graphs generated from latent positions traces back to the
work of Gilbert \cite{gilbert1961random} in the early $1960$s,
illustrated by applications to communication networks.
After several decades,
a latent space model with great generality was proposed in \cite{hoff2002latent}
and applied to social network analysis,
which popularized the modern study of latent space graphs in statistics.
The inner product model was later generalized to a latent position graph model equipped with
a continuous positive definite kernel in \cite{tang2013universally},
where feature map estimation and universally consistent vertex classification
were discussed.
A recent work \cite{ma2020universal} studied a similar inner product model
focusing on model fitting methods.
We refer to~\cite{athreya2018statistical} for a survey on such topics
in random dot product graphs.
Another family of random graphs with latent geometric structure is the random
geometric graphs \cite{penrose2003random}.
The soft variants
wherein the connection probability depends on the distance of latent positions
through a function were also investigated in \cite{penrose2016connectivity}.

The study of random geometric graphs in high dimensions originates from the
pioneering work of Devroye, Gy\"{o}rgy, Lugosi, and Udina
\cite{devroye2011high},
where they used a multivariate central limit theorem to show
that the graph becomes similar to an \ER{} graph when the dimension grows.
Subsequent work of Bubeck, Ding, Eldan, and R\'acz \cite{bubeck2016testing}
determined that the dimension threshold is $d \asymp n^{3}$
at which the
phase transition of losing geometry happens
in dense random geometric graphs.
In prior work, we
generalized this phase transition phenomena to
a noisy setting and studied the trade-off between noise and
dimensionality for the first time~\cite{liu2021phase};
this paper is the closest to the current one,
see Remark~\ref{rem:soft} above for a discussion.
An excellent recent survey~\cite{duchemin2021random} provides a detailed summary
of progress and discussions of open questions on these problems
(see also~\cite{RB17}).

Lying underneath the loss of geometry in random graphs is
the Wishart to GOE transition in high dimensions,
and a line of work explores this direction
\cite{bubeck2016testing,jiang2015approximation,bubeck2016entropic,racz2019smooth}.
In particular,
the phase transition was shown for log-concave measures using entropy-based
methods in~\cite{bubeck2016entropic}.
This was further extended to an anisotropic setting~\cite{eldan2020information}.
In a recent work, masked Wishart matrices were considered and phase transitions
were proven to matching orders in various types of
combinatorial masks~\cite{brennan2021finetti}.

\subsection{Organization}
The rest of the paper is organized as follows.
In Section~\ref{se:pre}, we introduce several notations and key facts used
throughout the paper.
Section~\ref{se:imp} consists of the proof of Theorem~\ref{th:imp},
where several information-theoretic inequalities are used.
Detecting geometry using the signed triangle statistic is presented in
Section~\ref{se:pos}, where the main body consists of estimating the expectation
of a signed triangle in $\cG(n,p,d,r)$ in two different parameter regimes.
Applying Chebyshev's inequality with these estimates concludes the proof of
Theorem~\ref{th:pos}.

\section{Notations and preliminaries} \label{se:pre}
A graph $G = (V, E)$ is a tuple consisting of a set of vertices
$V = [n] \coloneqq \{1,2,\ldots,n\}$ and
a set of edges $E \subset \binom{[n]}{2}$,
where the collection of all subsets of a set $S$ with cardinality $k$
is denoted by $\binom{S}{k}$.
We use $\norm{\cdot}$ to denote the Euclidean norm of a vector.
For a random variable $X \in \cX$ and a measurable function $f: \cX \to \RR$,
$\norm{f(X)}_p$ denotes the $L^p$-norm of $f \in L^p(\cX)$.

Our proofs build upon various inequalities involving
$f$-divergences, which are briefed here.
To begin with, we state the definition of an $f$-divergence.\footnote{Note
that the $f$ in $f$-divergence should not be confused
with the probability density function $f$ in the description
of $\cG(n,p,d,r)$. In general, $f$ is overloaded in this section,
but the meaning of $f$ will always be clear from the context.}
\begin{definition}[$f$-divergence] \label{de:fdiv}
Let $\cP$ and $\cQ$ be probability measures on the measurable space
$(\Omega, \cF)$.
Suppose that $\cP$ is absolutely continuous with respect to $\cQ$.
For a convex function $f$ such that $f(1) = 0$,
the $f$-divergence of $\cP$ and $\cQ$ is defined as
\begin{equation*}
\Df{\cP}{\cQ} \coloneqq \E_\cQ\biggl[f\biggl(\frac{d\cP}{d\cQ}\biggr)\biggr]
= \int_\Omega f\biggl(\frac{d\cP}{d\cQ}\biggr)\,d\cQ,
\end{equation*}
where $\frac{d\cP}{d\cQ}$ is the Radon--Nikodym derivative of $\cP$ with
respect to $\cQ$.
\end{definition}

Different choices of $f$
lead to the following $f$-divergences encountered in the proofs.
\begin{definition}
Let $\cP$ and $\cQ$ be probability measures on the measurable space
$(\Omega, \cF)$ such that $\cP$ is absolutely continuous with respect to $\cQ$.
\begin{thenum}
\item \emph{Total variation distance}
(corresponding to $f(x) = \frac{1}{2}\abs{x - 1}$):
\begin{equation*}
\tv{\cP}{\cQ} \coloneqq \sup_{A \in \cF} \abs{\cP(A)-\cQ(A)}
= \frac{1}{2} \int_\Omega \biggl\lvert \frac{d\cP}{d\cQ} - 1 \biggr\rvert\,d\cQ.
\end{equation*}
\item \emph{Kullback--Leibler (KL) divergence}
(corresponding to $f(x) = x \log x$):
\begin{equation*}
\kl{\cP}{\cQ} \coloneqq \E_\cP\biggl[\log\frac{d\cP}{d\cQ}\biggr]
= \int_\Omega \frac{d\cP}{d\cQ}\log\frac{d\cP}{d\cQ}\,d\cQ.
\end{equation*}
\item \emph{(Pearson) $\chi^2$-divergence} (corresponding to $f(x) = (x-1)^2$):
\begin{equation*}
\cs{\cP}{\cQ} \coloneqq \E_\cQ\biggl[\biggl(\frac{d\cP}{d\cQ}-1\biggr)^2\biggr]
= \E_\cQ\biggl[\biggl(\frac{d\cP}{d\cQ}\biggr)^2\biggr] - 1.
\end{equation*}
\end{thenum}
\end{definition}
The $f$-divergences defined above are connected through the following
inequalities (see~\cite{gibbs2002metrics}).

\begin{proposition}[Pinsker's inequality] \label{pp:pinsker}
Let $\cP$ and $\cQ$ be probability measures on the measurable space
$(\Omega, \cF)$ such that $\cP$ is absolutely continuous with respect to $\cQ$.
Then
\begin{equation*}
\tv{\cP}{\cQ}\le \sqrt{\frac{1}{2}\kl{\cP}{\cQ}}.
\end{equation*}
\end{proposition}

\begin{proposition} \label{pp:kl_chi-square}
Let $\cP$ and $\cQ$ be probability measures on the measurable space
$(\Omega, \cF)$ such that $\cP$ is absolutely continuous with respect to $\cQ$.
Then
\begin{equation*}
\kl{\cP}{\cQ} \le \log (1 + \cs{\cP}{\cQ}).
\end{equation*}
\end{proposition}

For KL divergence, a useful property is the chain rule, stated as the following
proposition.
\begin{proposition}[Chain rule] \label{pp:chain}
For joint distributions $\cP_{X, Y} = \cP_{X \mid Y} \cP_Y$
and $\cQ_{X, Y} = \cQ_{X \mid Y} \cQ_Y$,
the chain rule for the KL divergence reads
\begin{equation}
\kl{\cP_{X, Y}}{\cQ_{X, Y}}
= \kl{\cP_Y}{\cQ_Y} + \E_{\cP_Y}\kl{\cP_{X \mid Y}}{\cQ_{X \mid Y}}.
\end{equation}
\end{proposition}

We show the following lemma for Lebesgue integrable functions with bounded
derivatives.
\begin{lemma}\label{lm:der_bounded}
Suppose $f: \RR \to \RR$ is an integrable function on $\RR$
that is continuously differentiable.
If $f'$ is bounded, then $f$ is bounded as well.
Further,
let $M \coloneqq \int_\RR \abs{f(x)}\,dx < +\infty$
and $\alpha \coloneqq \sup_x\abs{f'(x)}$;
then $\sup_x \abs{f(x)} \le 2\sqrt{M\alpha}$.
\end{lemma}
\begin{proof}
For all $a \in \RR$ and $b > 0$, consider the integral
\begin{equation*}
\int_a^{a+b}\abs{f(x)}\,dx \le \int_\RR \abs{f(x)}\,dx = M.
\end{equation*}
Let $f_m \coloneqq \inf_{x \in [a,a+b]} \abs{f(x)}$.
Since $f_m \le \abs{f(x)}$ for $x \in [a,a+b]$,
\begin{equation*}
b f_m = \int_a^{a+b}f_m\,dx \le \int_a^{a+b}\abs{f(x)}\,dx \le M,
\end{equation*}
which gives $f_m \le M/b$.

Since $f$ is continuous,
there exists a number $c \in [a, a+b]$ such that $\abs{f(c)} = f_m$.
The mean value theorem gives that for $\xi \in [a,a+b]$,
\begin{equation*}
f(a) \le f_m + \abs{f'(\xi)}b \le \frac{M}{b} + \alpha b.
\end{equation*}
If $\alpha > 0$, by choosing $b = \sqrt{M/\alpha}$, the claim directly follows.

If $\alpha = 0$, then $f$ is constant on $\RR$.
Since $f$ is integrable, $f(x) = 0$ for all $x \in \RR$.
Thus, the claim is also true.
\end{proof}
Applying Lemma~\ref{lm:der_bounded} to a probability density function,
we have the following corollary.
\begin{corollary} \label{cr:bounded}
Let $f: \RR \to \RR_+$ be a probability density function
that is continuously differentiable.
If $\sup_x\abs{f'(x)} \le \alpha$,
then $\sup_x \abs{f(x)} \le 2\sqrt{\alpha}$.
\end{corollary}

For standard normal random variables, Stein's lemma
(also known as Gaussian integration by parts)
provides a powerful tool and
is frequently used in the proofs.
The lemma, and Stein's method built upon it, are broadly used in probability
and statistics (see, e.g., \cite[Example~13.13]{janson1997gaussian},
also \cite{chen2011stein}).
We state it as the following proposition.
\begin{proposition}(Stein's lemma \cite[Lemma~1]{stein1981mean})
Let $Y$ be a $\cN(0,1)$ real random variable and let $g: \RR \to \RR$
be an indefinite integral of the Lebesgue measurable function $g'$,
essentially the derivative of $g$.
Suppose also that $\E[\abs{g'(Y)}] < \infty$.
Then,
\begin{equation*}
\E[g'(Y)] = \E[Yg(Y)].
\end{equation*}
\end{proposition}

For a continuously differentiable function of a standard normal random vector,
the following proposition provides a sharp bound for the variance,
known as the Gaussian Poincar\'e inequality
(see \cite[Theorem~3.20]{boucheron2013concentration}).
\begin{proposition}[Gaussian Poincar\'e inequality] \label{pp:poincare}
Suppose $\x = (x_1, \ldots, x_d)$ is a vector of i.i.d. standard Gaussian random
variables.
Let $f: \RR^d \to \RR$ be any continuously differentiable function.
Then,
\begin{equation*}
\Var[f(\x)] \le \E[\norm{\nabla f(\x)}^2].
\end{equation*}
\end{proposition}

We also make frequent use of the properties of sub-exponential random variables.
We state the definition in terms of the moment generating function and
the equivalent tail bound,
which appear in most texts (see, e.g.,
\cite[Definition~2.7 and Proposition~2.9]{wainwright2019high}).
\begin{definition} \label{df:subexp}
A random variable $X$ is \emph{sub-exponential} if there are nonnegative
parameters $(a, b)$ such that
\begin{equation*}
\log \E[e^{t(X-\E[X])}] \le \frac{a^2 t^2}{2}
\qquad \text{for all $\abs{t} < \frac{1}{b}$}.
\end{equation*}
\end{definition}

\begin{proposition} \label{pp:subexp}
Suppose $X$ is sub-exponential with parameters $(a, b)$.
Then,
\begin{equation*}
\P(X-\E[X] \ge t) \le \begin{cases}
\exp\bigl(-\frac{t^2}{2a^2}\bigr) \quad \text{if $0 \le t \le \frac{a^2}{b}$},\\
\exp\bigl(-\frac{t}{2b}\bigr) \quad \text{for $t > \frac{a^2}{b}$}.
\end{cases}
\end{equation*}
Equivalently,
\begin{equation*}
\P(X-\E[X] \ge t) \le
\exp\biggl(-\frac{1}{2}\min\biggl\lbrace
\frac{t^2}{a^2}, \frac{t}{b}
\biggr\rbrace\biggr).
\end{equation*}
\end{proposition}

The following concentration lemma regarding the normal distribution
is used in various proofs.
\begin{lemma} \label{lm:cct_norm_func}
Let $f: \RR^d \to \RR$ be a continuously differentiable function.
If the norm of the gradient satisfies
$\norm{\nabla f(\x)} \le \norm{\x}$,
then for $\x \sim \cN(\0, \I_d)$,
$f(\x)$ is sub-exponential with parameters $(2\sqrt{d},1)$.
In particular, the tails of $f(\x)$ satisfy
\begin{equation*}
\P(\abs{f(\x) - \E[f(\x)]} \ge t)
\le \begin{cases}
2\exp\bigl(-\frac{t^2}{8d}\bigr) &\text{if $0 \le t \le 4d$},\\
2\exp\bigl(-\frac{t}{2}\bigr) &\text{for $t > 4d$}.
\end{cases}
\end{equation*}
\end{lemma}
\begin{proof}
Without loss of generality, we may assume $\E[f(\x)] = 0$.

For a nonnegative random variable $Z$,
the \emph{entropy} of $Z$ is defined as
\begin{equation*}
\Ent(Z) \coloneqq \E[\varphi(Z)] - \varphi(\E[Z]),
\end{equation*}
where $\varphi(x) = x \log x$ (see, e.g.,~\cite{boucheron2013concentration};
note that this notion of entropy is not to be confused with the Shannon
entropy).

By the Gaussian logarithmic Sobolev inequality
(see, e.g., \cite[Theorem~5.4]{boucheron2013concentration}),
\begin{equation*}
\Ent(e^{tf(\x)}) \le 2\E[\norm{\nabla e^{tf(\x)/2}}^2]
= \frac{t^2}{2}\E[e^{tf(\x)}\norm{\nabla f(\x)}^2]
\le \frac{t^2}{2}\E[\norm{\x}^2 e^{tf(\x)}].
\end{equation*}

The duality formula of the entropy
(see \cite[Remark~4.4]{boucheron2013concentration})
implies that for any random variable~$W$ such that $\E[e^W] < \infty$,
the entropy of $e^{tZ}$ for a random variable $Z$ satisfies
\begin{equation*}
\E[W e^{t Z}]
\le \E[e^{t Z}]\log\E[e^{W}] + \Ent(e^{t Z}).
\end{equation*}
Applying the inequality with $Z=f(\x)$ and $W=\frac{e-1}{2e}\norm{\x}^2$,
we have that
\begin{equation*}
\E[\norm{\x}^2e^{t f(\x)}]
\le \frac{2e}{e-1}\E[e^{t f(\x)}]
\log\E\biggl[\exp\biggl(\frac{e-1}{2e}\norm{\x}^2\biggr)\biggr]
+ \Ent(e^{tf(\x)}).
\end{equation*}

Since $\x \sim \cN(\0, \I_d)$, $\norm{\x}^2$ has a chi-squared distribution
with $d$ degrees of freedom.
By the moment generating function of the chi-squared distribution, we have that
\begin{equation*}
\log\E\biggl[\exp\biggl(\frac{e-1}{2e}\norm{\x}^2\biggr)\biggr]
= \log \biggl(1 - 2 \cdot \frac{e-1}{2e}\biggr)^{-d/2}
= \frac{d}{2}.
\end{equation*}

Hence, by putting together the previous displays,
we have that
\begin{equation*}
\Ent(e^{tf(\x)})
\le \frac{t^2}{2}\biggl(\frac{ed}{e-1}\E[e^{tf(\x)}] + \Ent(e^{tf(\x)})\biggr)
\le \frac{t^2}{2}\biggl(2d\E[e^{tf(\x)}] + \Ent(e^{tf(\x)})\biggr).
\end{equation*}
By rearranging this inequality, we have that
\begin{equation*}
\Ent(e^{tf(\x)})
\leq \frac{d t^{2}}{1-t^{2}/2} \E[e^{tf(\x)}]
\le 2d t^2 \E[e^{tf(\x)}],
\end{equation*}
where the second inequality holds for $|t| \leq 1$.
Writing $M(t) \coloneqq \E[e^{tf(\x)}]$,
note that
$\Ent(e^{tf(\x)}) = t M'(t) - M(t) \log M(t)$.
Therefore the inequality in the previous display becomes
\begin{equation*}
t M'(t) - M(t) \log M(t)
\le 2d t^2 M(t).
\end{equation*}
Solving this equation exactly, and noting that $M(0)=1$, gives that
\begin{equation*}
\log \E[e^{tf(\x)}] \le 2dt^2.
\end{equation*}
By Definition~\ref{df:subexp}, the claim is hence proved.

Further by Proposition~\ref{pp:subexp}, the tail bound directly follows.
\end{proof}

\begin{remark}
This lemma can also be derived from an exponential Poincar\'e inequality
\cite[problem~3.16]{vanhandel2014probability},
as pointed out by Ramon van Handel (personal communication).
\end{remark}

The following proposition characterizes the tail behavior of the inner product
of two independent $d$-dimensional standard normal random vectors.
The sub-exponential tails in the proposition can be derived from
Lemma~\ref{lm:cct_norm_func};
however, since the function is explicit,
we prove it directly using the moment generating function,
resulting in slightly different parameters.

\begin{proposition} \label{pp:inner_gauss_cct}
For $\x, \y \in \RR^d$ independently distributed as $\cN(\0, \I_d)$,
the inner product $\inner{\x}{\y}$ is sub-exponential with parameters
$(\sqrt{2d}, \sqrt{2})$.
\end{proposition}

\begin{proof}
Consider two independent random variables $X, Y \sim \cN(0, 1)$.
The moment generating function of their product satisfies
\begin{equation*}
\E[e^{t XY}]
= \iint \frac{1}{2\pi}e^{t xy}e^{-(x^2+y^2)/2}\,dx\,dy
= \begin{vmatrix}
1 & -t\\
-t & 1
\end{vmatrix}^{-1/2}
= \frac{1}{\sqrt{1-t^2}}.
\end{equation*}
Since $(1-x)^{-1/2} \le e^{x}$ for $0 \le x \le 1/2$,
we have for $t^2 \le 1/2$ that
$\E[e^{t XY}] \le e^{t^2}$.

Consequently,
for two independent standard normal random vectors $\x, \y \sim \cN(\0, \I_d)$,
we have that
\begin{equation*}
\E[e^{t\inner{\x}{\y}}] = \E[e^{t\sum_{i=1}^d x_iy_i}]
= \prod_{i=1}^d \E[e^{t x_i y_i}] \le e^{d t^2}
\end{equation*}
for $t^2 \le 1/2$.
\end{proof}

We state a lemma concerning the concentration of
the inner product of two independent random vectors uniformly distributed
on the unit sphere $\SS^{d-1}$.
\begin{lemma} \label{lm:sphere_inner_cct}
For $\x,\y \in \RR^d$ independently uniformly distributed on the unit sphere
$\SS^{d-1}$, when $t \ge 1$ and $d \ge 2$,
\begin{equation*}
\P\biggl(\abs{\inner{\x}{\y}} \ge \frac{t}{\sqrt{d}}\biggr)
\le 2\exp\biggl(-\frac{t^2}{4}\biggr).
\end{equation*}
\end{lemma}
\begin{proof}
By rotation invariance on the $d$-dimensional sphere,
we can fix $\y = \e_1 \coloneqq (1, 0, \ldots, 0)$,
the first vector of the standard basis.
Let $\z \in \cN(\0,\I_d)$, then $\z/\norm{\z}$ is a uniform random point in
$\SS^{d-1}$.
Therefore, we have that
\begin{equation*}
\P\biggl(\abs{\inner{\x}{\y}} \ge \frac{t}{\sqrt{d}}\biggr)
= \P\biggl(\frac{\abs{z_1}}{\norm{\z}} \ge \frac{t}{\sqrt{d}}\biggr)
= \P\biggl(\frac{z_1^2}{\sum_{i=1}^d z_i^2} \ge \frac{t^2}{d}\biggr).
\end{equation*}
Since $z_i$'s are i.i.d. standard normal random variables,
$z_1^2/\sum_{i=1}^d z_i^2$ has a $\mathrm{Beta}(\frac{1}{2},\frac{d-1}{2})$
distribution.
Hence, by the density function of a beta distribution, we have that
\begin{equation*}
\P\biggl(\frac{z_1^2}{\sum_{i=1}^d z_i^2} \ge \frac{t^2}{d}\biggr)
=\frac{\Gamma(\frac{d}{2})}{\Gamma(\frac{1}{2})\Gamma(\frac{d-1}{2})}
\int_{\frac{t^2}{d}}^1 z^{-1/2}(1-z)^{(d-1)/2-1}\,dz.
\end{equation*}
By Wendel's double inequality (see \cite[equation (7)]{wendel1948note},
also \cite[equation (3.15)]{liu2021phase}), we have that
\begin{equation*}
\frac{\Gamma(\frac{d}{2})}{\Gamma(\frac{1}{2})\Gamma(\frac{d-1}{2})}
\le \sqrt{\frac{d-1}{2\pi}}.
\end{equation*}
Additionally,
\begin{equation*}
\begin{split}
\int_{\frac{t^2}{d}}^1 z^{-1/2}(1-z)^{(d-1)/2-1}\,dz
&\le \frac{\sqrt{d}}{t}\int_{\frac{t^2}{d}}^1 (1-z)^{(d-1)/2-1}\,dz
= \frac{\sqrt{d}}{t}
\biggl(-\frac{2}{d-1}(1-z)^{(d-1)/2}\biggr)\bigg\vert_{\frac{t^2}{d}}^1\\
&= \frac{2\sqrt{d}}{t(d-1)}\biggl(1-\frac{t^2}{d}\biggr)^{(d-1)/2}
\le \frac{2\sqrt{d}}{t(d-1)}\exp\biggl(-\frac{(d-1)t^2}{2d}\biggr).
\end{split}
\end{equation*}

Putting the previous displays together, we have for $t \ge 1$ and $d \ge 2$ that
\begin{equation*}
\P\biggl(\abs{\inner{\x}{\y}} \ge \frac{t}{\sqrt{d}}\biggr)
\le \frac{2}{\sqrt{\pi}}\exp\biggl(-\frac{t^2}{4}\biggr).
\end{equation*}
The claim directly follows.
\end{proof}

\begin{remark}
The proof of Lemma~\ref{lm:sphere_inner_cct} makes use of the explicit
density function of a Beta distribution.
The tails of Beta random variables have also been studied in recent years
\cite{marchal2017beta,zhang2020non},
where sub-Gaussian and Bernstein-type bounds are given respectively in
different parameter regimes.
\end{remark}

\begin{remark}
We have seen in Proposition~\ref{pp:inner_gauss_cct}
that the inner product of two $d$-dimensional standard normal random vectors
has sub-exponential tails, and since the moment generating function does not
exist when $t > 1$, the exponential rate cannot be improved.
In comparison, Lemma~\ref{lm:sphere_inner_cct} gives sub-Gaussian tails
when the random vectors are uniformly distributed on a sphere of
the same dimension,
which decays much faster.
In other words, the inner product of independent high dimensional random vectors
concentrates better on a sphere than in a Gaussian space.
\end{remark}

\section{Impossibility of detecting geometry} \label{se:imp}
In this section, we show that $\cG(n,p,d,r)$ and $\cG(n,p)$ are
indistinguishable when the dimension $d$ or the parameter $r$ is large,
thus proving Theorem~\ref{th:imp}.

We first introduce several notations used in the proofs.
We denote the adjacency matrix of $\cG(n,p,d,r)$ by $\A$.
Let $\B \in \RR^{n \times n}$ be a symmetric Bernoulli ensemble, that is,
$\{b_{i,j}\}_{1 \le i< j \le n}$ are independent Bernoulli random variables
with parameter $p$.
We also use the following shorthand notations in this section.
For the matrices $\A, \B \in \RR^{n \times n}$,
we denote their principal minor of order~$k$ by $\A_k$ and $\B_k$.
The bold lower case letter $\a_k$ denotes the last row of $\A_k$.
The $k$ by $d$ matrix consisting of the first $k$ rows of
the matrix $\X \in \RR^{n \times d}$ is denoted
by $\X_k$, and $\x_k$ denotes the $k$th row of~$\X$.

Pinsker's inequality (Proposition~\ref{pp:pinsker}) gives
\begin{equation*}
\tv{\cG(n,p,d,r)}{\cG(n,p)} \le \sqrt{\frac{1}{2}\kl{\cG(n,p,d,r)}{\cG(n,p)}}.
\end{equation*}
By the chain rule of KL divergence (Proposition~\ref{pp:chain}) we have that
\begin{equation} \label{eq:chain}
\begin{split}
\kl{\cG(n,p,d,r)}{\cG(n,p)} = \kl{\A}{\B}
&= \sum_{k=0}^{n-1} \E_{\A_k} \kl{\a_{k+1} \mid \A_k}{\b_{k+1} \mid \B_k = \A_k}\\
&= \sum_{k=0}^{n-1} \E_{\A_k} \kl{\a_{k+1} \mid \A_k}{\b_{k+1}},
\end{split}
\end{equation}
where the last equality is due to the independence of $\b_{k+1}$ and $\B_k$.
By convexity of the KL divergence
(see, e.g., \cite[Proposition~3.4]{liu2021phase}),
Jensen's inequality gives that
\begin{equation*}
\begin{split}
\E_{\A_k} \kl{\a_{k+1} \mid \A_k}{\b_{k+1}}
&\le \E_{\A_k, \X_k} \kl{\a_{k+1} \mid \A_k, \X_k}{\b_{k+1}}\\
&= \E_{\X_k} \kl{\a_{k+1} \mid \X_k}{\b_{k+1}},
\end{split}
\end{equation*}
where the last equality is because $\a_{k+1}$ and $\A_{k}$ are conditionally
independent given $\X_k$.
The KL divergence is bounded from above by the $\chi^2$
divergence (see Proposition~\ref{pp:kl_chi-square}):
\begin{equation} \label{eq:kl_chi_square_log}
\begin{split}
\E_{\X_k} \kl{\a_{k+1} \mid \X_k}{\b_{k+1}}
&\le \E_{\X_k}[\log(1+\cs{\a_{k+1} \mid \X_k}{\b_k})]\\
&\le \log(1+\E_{\X_k}\cs{\a_{k+1} \mid \X_k}{\b_k})\\
&= \log \E_{\X_k, \b_{k+1}}\biggl[
\biggl(\frac{\P_{\cG(n,p,d,r)}(\b_{k+1} \mid \X_k)}
{\P_{\cG(n,p)}(\b_{k+1})}\biggr)^2\biggr],
\end{split}
\end{equation}
where the second line is by Jensen's inequality.

By the definition of $\cG(n,p,d,r)$, we have that
\begin{equation*}
\P_{\cG(n,p,d,r)}(\b_{k+1} \mid \X_k) = \E_{\x_{k+1}}\biggl[
\prod_{1 \le i \le k}\sigma(\inner{\x_i}{\x_{k+1}})^{b_{i,k+1}}
(1 - \sigma(\inner{\x_i}{\x_{k+1}}))^{1-b_{i,k+1}}\biggr].
\end{equation*}
By an idea similar to the second moment method \cite{brennan2020phase},
we can write the square as the product of two expectations of independent copies
and then apply Fubini's theorem:
\begin{equation*}
\begin{split}
\P_{\cG(n,p,d,r)}(\b_{k+1} \mid \X_k)^2
&= \E_{\x_{k+1}}\biggl[
\prod_{1 \le i \le k}\sigma(\inner{\x_i}{\x_{k+1}})^{b_{i,k+1}}
(1 - \sigma(\inner{\x_i}{\x_{k+1}}))^{1-b_{i,k+1}}\biggr]\\
&\phantom{{}={}} \times\E_{\x_{k+1}'}\biggl[
\prod_{1 \le i \le k}\sigma(\inner{\x_i}{\x_{k+1}'})^{b_{i,k+1}}
(1 - \sigma(\inner{\x_i}{\x_{k+1}'}))^{1-b_{i,k+1}}\biggr]\\
&= \E_{\x_{k+1},\x_{k+1}'}\biggl[
\prod_{1 \le i \le k}
\Bigl(\sigma(\inner{\x_i}{\x_{k+1}})\sigma(\inner{\x_i}{\x_{k+1}'})\Bigr)
^{b_{i,k+1}}\\
&\phantom{{}={}}\times\Bigl((1 - \sigma(\inner{\x_i}{\x_{k+1}}))
(1 - \sigma(\inner{\x_i}{\x_{k+1}'}))\Bigr)^{1-b_{i,k+1}}\biggr],
\end{split}
\end{equation*}
where the last equality is by the independence of $\x_{k+1}$ and $\x'_{k+1}$.
Therefore, by interchanging the expectations and taking out the product by
independence, we obtain that
\begin{equation*}
\begin{split}
&\E_{\X_k, \b_{k+1}}\biggl[\biggl(\frac{\P_{\cG(n,p,d,r)}(\b_{k+1} \mid \X_k)}
{\P_{\cG(n,p)}(\b_{k+1})}\biggr)^2\biggr]\\
&\qquad= \E_{\x_{k+1},\x_{k+1}'}\biggl[
\prod_{1 \le i \le k} \E_{\x_i,b_{i,k+1}}\biggl[
\biggl(\frac{1}{p^2}\sigma(\inner{\x_i}{\x_{k+1}})
\sigma(\inner{\x_i}{\x_{k+1}'})\biggr)^{b_{i,k+1}}\\
&\qquad\phantom{{}={}}\times\biggl(\frac{1}{(1-p)^2}
(1 - \sigma(\inner{\x_i}{\x_{k+1}}))
(1 - \sigma(\inner{\x_i}{\x_{k+1}'}))\biggr)^{1-b_{i,k+1}}\biggr]\biggr].
\end{split}
\end{equation*}
Since each entry of $\B$ is an independent Bernoulli random variable with
parameter $p$, we can compute the inner expectation over $b_{i,k+1}$ directly,
obtaining that
\begin{equation} \label{eq:chi_square_k}
\begin{split}
&\E_{\X_k, \b_{k+1}}\biggl[\biggl(\frac{\P_{\cG(n,p,d,r)}(\b_{k+1} \mid \X_k)}
{\P_{\cG(n,p)}(\b_{k+1})}\biggr)^2\biggr]\\
&\qquad= \E_{\x_{k+1},\x_{k+1}'}\biggl[
\prod_{1 \le i \le k}\biggl(1+\frac{1}{p(1-p)}\E_{\x_i}
[(\sigma(\inner{\x_i}{\x_{k+1}})-p)(\sigma(\inner{\x_i}{\x_{k+1}'})-p)]\biggr)
\biggr]\\
&\qquad= \E_{\x_{k+1},\x_{k+1}'}\biggl[
\biggl(1+\frac{1}{p(1-p)}\E_{\x_1}
[(\sigma(\inner{\x_1}{\x_{k+1}})-p)(\sigma(\inner{\x_1}{\x_{k+1}'})-p)]\biggr)^k
\biggr],
\end{split}
\end{equation}
where the last equality holds since all $\x_i$'s are identically distributed.

Define
\begin{equation} \label{eq:gamma_def}
\gamma(\x,\y)
\coloneqq \E_{\z\sim\cN(\0,\I_d)}[(\sigma(\inner{\x}{\z})-p)
(\sigma(\inner{\y}{\z})-p)].
\end{equation}
We show the following lemma concerning $\gamma(\x,\y)$.
\begin{lemma} \label{lm:gamma_results}
Let $\x, \y \in \RR^d$ be independent standard normal random vectors.
Recall the definition of~$\sigma$,
as well as the assumption~\eqref{eq:A1}.
\begin{thenum}
\item \label{lm:gamma_E} The mean of $\gamma(\x,\y)$ satisfies
\begin{equation*}
0 \le \E[\gamma(\x,\y)] \le \frac{\alpha^2}{r^4d}.
\end{equation*}
\item \label{lm:gamma_Var} The variance of $\gamma(\x,\y)$ is upper bounded by
\begin{equation*}
\Var[\gamma(\x,\y)] \le \frac{68\alpha^2}{r^4d}.
\end{equation*}
\item \label{lm:gamma_tail}
Let $L \coloneqq \sqrt{34}\alpha/(r^2d)$.
Then, $\gamma(\x,\y)/L$ is sub-exponential with
parameters $(2\sqrt{2d},1)$, that is,
\begin{equation*}
\log\E\biggl[\exp\biggl(\frac{t}{L}(\gamma(\x,\y)-\E[\gamma(\x,\y)])\biggr)
\biggr]
\le 4 d t^2
\quad \text{for all $\abs{t} \le 1$}.
\end{equation*}
\end{thenum}
\end{lemma}

\begin{proof}[Proof of Lemma~\ref{lm:gamma_E}]
Let
\begin{equation*}
\eta(\x) \coloneqq \E_{\z\sim\cN(0,\I_d)}[\sigma(\inner{\x}{\z})]
\qquad\text{and}\qquad
\xi(\x,\y) \coloneqq
\E_{\z\sim\cN(0,\I_d)}[\sigma(\inner{\x}{\z})\sigma(\inner{\y}{\z})].
\end{equation*}
By the construction of $\sigma$,
we have that
\begin{equation}\label{eq:eta_exp}
\E_{\x\sim\cN(0,\I_d)}[\eta(\x)] = \E[\sigma(\inner{\x}{\z})] = p.
\end{equation}
Next, we bound the variance of $\eta(\x)$. Observe that
\begin{equation*}
\frac{\partial \eta(\x)}{\partial x_i}
= \E_{\z\sim\cN(\0,\I_d)}\biggl[
\frac{\partial \sigma(\inner{\x}{\z})}{\partial x_i}\biggr]
= \E_{\z\sim\cN(\0,\I_d)}[z_i \sigma'(\inner{\x}{\z})].
\end{equation*}
Thus by Stein's lemma we have that
\begin{equation*}
\begin{split}
\frac{\partial \eta(\x)}{\partial x_i}
= \E_{\z_{-i}}[\E_{z_i}[z_i\sigma'(\inner{\x}{\z})]]
= \E_{\z_{-i}}\biggl[\E_{z_i}\biggl[
\frac{\partial \sigma'(\inner{\x}{\z})}{\partial z_i}\biggr]\biggr]
= x_i \E_{\z}[\sigma''(\inner{\x}{\z})],
\end{split}
\end{equation*}
where $\z_{-i}$ denotes the rest of $\z$ except for $z_i$.
Hence, by the definition of $\sigma$ in \eqref{eq:sigma},
\begin{equation*}
\begin{split}
\norm{\nabla \eta(\x)}^2
&= \sum_{i=1}^d\biggl(\frac{\partial \eta(\x)}{\partial x_i}\biggr)^2
= \frac{1}{r^4d^2}\sum_{i=1}^d x_i^2 \E_{\z}\biggl[
f'\biggl(\frac{\inner{\x}{\z}-\mu_{p,d,r}}{r\sqrt{d}}\biggr)\biggr]^2\\
&= \frac{\norm{\x}^2}{r^4d^2}\E_{\z}\biggl[
f'\biggl(\frac{\inner{\x}{\z}-\mu_{p,d,r}}{r\sqrt{d}}\biggr)\biggr]^2
\le \frac{\norm{\x}^2}{r^4d^2}\E_{\z}\biggl[
f'\biggl(\frac{\inner{\x}{\z}-\mu_{p,d,r}}{r\sqrt{d}}\biggr)^2\biggr]
\le \frac{\alpha^2\norm{\x}^2}{r^4d^2},
\end{split}
\end{equation*}
where we used the assumption~\eqref{eq:A1} in the last inequality.
The Gaussian Poincar\'e inequality (Proposition~\ref{pp:poincare}) thus gives
\begin{equation}\label{eq:eta_var}
\Var[\eta(\x)]
\le \E[\norm{\nabla \eta(\x)}^2]
\le \frac{\alpha^2}{r^4d^2}\E[\norm{\x}^2]
= \frac{\alpha^2}{r^4d}.
\end{equation}
Additionally, by interchanging the order of expectations we have that
\begin{equation}\label{eq:gamma_exp}
\E[\xi(\x,\y)]
= \E[\E_{\z}[\sigma(\inner{\x}{\z})\sigma(\inner{\y}{\z})]]
= \E[\E_{\x}[\sigma(\inner{\x}{\z})]\E_{\y}[\sigma(\inner{\y}{\z})]]
= \E[\eta(\z)^2].
\end{equation}
Finally, expanding the product in the definition of $\gamma(\x,\y)$
and putting together the expressions in~\eqref{eq:eta_exp},~\eqref{eq:eta_var},
and~\eqref{eq:gamma_exp},
we obtain that
\begin{equation} \label{eq:gamma_upper}
\E[\gamma(\x,\y)]
= \E[\xi(\x,\y)] - p\E[\eta(\x)] - p\E[\eta(\y)] + p^2
= \E[\xi(\x,\y)] - p^2
= \Var[\eta(\z)]
\le \frac{\alpha^2}{r^4d}.
\end{equation}
The nonnegativity of $\E[\gamma(\x,\y)]$ directly follows from it being a
variance.
\end{proof}

Recall the assumption~\eqref{eq:A1},
and note that
Corollary~\ref{cr:bounded} thus implies that $f$ is bounded and
\begin{equation} \label{eq:f_bound}
\sup_x\abs{f(x)} \le 2\sqrt{\alpha}.
\end{equation}

\begin{proof}[Proof of Lemma~\ref{lm:gamma_Var} and Lemma~\ref{lm:gamma_tail}]
Taking the partial derivative with respect to $x_i$, we have that
\begin{equation*}
\frac{\partial \gamma(\x, \y)}{\partial x_i}
= \E_{\z\sim\cN(0,\I_d)}\biggl[
\frac{\partial(\sigma(\inner{\x}{\z})-p)}{\partial x_i}
(\sigma(\inner{\y}{\z})-p)\biggr]
= \E_{\z\sim\cN(0,\I_d)}[z_i\sigma'(\inner{\x}{\z})(\sigma(\inner{\y}{\z})-p)].
\end{equation*}
By Stein's lemma, we have that
\begin{equation*}
\begin{split}
\frac{\partial \gamma(\x, \y)}{\partial x_i}
&= \E_{\z_{-i}}[\E_{z_i}[z_i\sigma'(\inner{\x}{\z})(\sigma(\inner{\y}{\z})-p)]]
= \E_{\z_{-i}}\biggl[\E_{z_i}\biggl[\frac{\partial}{\partial z_i}
(\sigma'(\inner{\x}{\z})(\sigma(\inner{\y}{\z})-p))\biggr]\biggr]\\
&= y_i \E_{\z}[\sigma'(\inner{\x}{\z})\sigma'(\inner{\y}{\z})]
+ x_i \E_{\z}[\sigma''(\inner{\x}{\z})(\sigma(\inner{\y}{\z})-p)].
\end{split}
\end{equation*}
Then, by the elementary inequality $(a + b)^2 \le 2(a^2 + b^2)$
and Jensen's inequality, we have that
\begin{equation*}
\begin{split}
\biggl(\frac{\partial \gamma(\x, \y)}{\partial x_i}\biggr)^2
&\le \frac{2y_i^2}{r^4 d^2}\E_{\z}\biggl[
f\biggl(\frac{\inner{\x}{\z}-\mu_{p,d,r}}{r\sqrt{d}}\biggr)
f\biggl(\frac{\inner{\y}{\z}-\mu_{p,d,r}}{r\sqrt{d}}\biggr)\biggr]^2\\
&\phantom{{}\le{}}+\frac{2x_i^2}{r^4 d^2}
\E_{\z}\biggl[f'\biggl(\frac{\inner{\x}{\z}-\mu_{p,d,r}}{r\sqrt{d}}\biggr)
(\sigma(\inner{\y}{\z})-p)\biggr]^2\\
&\le \frac{2y_i^2}{r^4 d^2}\E_{\z}\biggl[
f\biggl(\frac{\inner{\x}{\z}-\mu_{p,d,r}}{r\sqrt{d}}\biggr)^2
f\biggl(\frac{\inner{\y}{\z}-\mu_{p,d,r}}{r\sqrt{d}}\biggr)^2\biggr]\\
&\phantom{{}\le{}}+\frac{2x_i^2}{r^4 d^2}
\E_{\z}\biggl[f'\biggl(\frac{\inner{\x}{\z}-\mu_{p,d,r}}{r\sqrt{d}}\biggr)^2
(\sigma(\inner{\y}{\z})-p)^2\biggr]\\
&\le \frac{2\alpha^2(x_i^2 + 16y_i^2)}{r^4d^2},
\end{split}
\end{equation*}
where in the last inequality we used~\eqref{eq:A1} and~\eqref{eq:f_bound}.
Hence, we obtain that
\begin{equation} \label{eq:gamma_grad_norm}
\norm{\nabla\gamma(\x,\y)}^2
= \sum_{i=1}^d \biggl(\frac{\partial\gamma(\x,\y)}{\partial x_i}\biggr)^2
+ \sum_{i=1}^d \biggl(\frac{\partial\gamma(\x,\y)}{\partial y_i}\biggr)^2
\le \frac{34\alpha^2}{r^4d^2}(\norm{\x}^2+\norm{\y}^2).
\end{equation}
Taking expectation on both sides of the above display yields
\begin{equation*}
\E[\norm{\nabla\gamma(\x,\y)}^2]
\le \frac{34\alpha^2}{r^4d^2}\E[\norm{\x}^2+\norm{\y}^2]
= \frac{68\alpha^2}{r^4d}.
\end{equation*}
By the Gaussian Poincar\'e inequality we thus have that
\begin{equation*}
\Var[\gamma(\y,\z)] \le \E[\norm{\nabla\gamma(\y,\z)}^2]
\le\frac{68\alpha^2}{r^4d}.
\end{equation*}
Lemma~\ref{lm:gamma_Var} is hence proved.

By viewing $(\x, \y)$ as a $(2d)$-dimensional vector,
\eqref{eq:gamma_grad_norm} exactly gives the upper bound of the norm of the
gradient in terms of the norm of the vector.
Thus, by applying Lemma~\ref{lm:cct_norm_func},
the sub-exponential tails of $\gamma(\x,\y)$ in Lemma~\ref{lm:gamma_tail}
directly follow.
\end{proof}

With Lemma~\ref{lm:gamma_results} in place, we return to bounding the
KL divergence from above.
Using the definition of $\gamma(\x,\y)$,
we can express \eqref{eq:chi_square_k} with $\gamma(\x,\y)$ as
\begin{equation*}
\E_{\X_k,\b_{k+1}}\biggl[\biggl(\frac{\P_{\cG(n,p,d,r)}(\b_{k+1} \mid \X_k)}
{\P_{\cG(n,p)}(\b_{k+1})}\biggr)^2\biggr]
= \E\biggl[\biggl(1+\frac{1}{p(1-p)}\gamma(\x,\y)\biggr)^k\biggr]
\le \E\biggl[\exp\biggl(\frac{k}{p(1-p)}\gamma(\x,\y)\biggr)\biggr],
\end{equation*}
where we use the fact that $1+x \le \exp(x)$.
Using Lemma~\ref{lm:gamma_tail} with $t = Lk/(p(1-p))$,
we have, for $r^2d/k \ge \sqrt{34}\alpha/(p(1-p))$, that
\begin{equation*}
\E\biggl[\exp\biggl(\frac{k}{p(1-p)}
(\gamma(\x,\y)-\E[\gamma(\x,\y)])\biggr)\biggr]
\le \exp\biggl(\frac{136\alpha^2}{p^2(1-p)^2}\cdot\frac{k^2}{r^4d}\biggr).
\end{equation*}
Recall that we assume that $n^{3}/(r^{4}d) \to 0$,
which implies that $r^{2}d/n \to \infty$,
so the inequality in the display above holds eventually
(uniformly for all $k \leq n$).
Combined with Lemma~\ref{lm:gamma_E}, we thus have that
\begin{equation*}
\begin{split}
\E\biggl[\exp\biggl(\frac{k}{p(1-p)}\gamma(\x,\y)\biggr)\biggr]
&= \exp\biggl(\frac{k}{p(1-p)}\E[\gamma(\x,\y)]\biggr)
\E\biggl[\exp\biggl(\frac{k}{p(1-p)}
(\gamma(\x,\y)-\E[\gamma(\x,\y)])\biggr)\biggr]\\
&\le \exp\biggl(\frac{\alpha^2}{p(1-p)}\cdot\frac{k}{r^4d}
+ \frac{136\alpha^2}{p^2(1-p)^2}\cdot\frac{k^2}{r^4d}\biggr).
\end{split}
\end{equation*}

Putting the previous displays together and inserting the upper bound into
\eqref{eq:kl_chi_square_log},
we obtain that
\begin{equation*}
\E_{\X_k} \kl{\a_{k+1} \mid \X_k}{\b_{k+1}}
\le \frac{\alpha^2}{p(1-p)}\cdot\frac{k}{r^4d}
+ \frac{136\alpha^2}{p^2(1-p)^2}\cdot\frac{k^2}{r^4d}.
\end{equation*}
Plugging the above display into \eqref{eq:chain}, we conclude that
\begin{equation*}
\begin{split}
\kl{\cG(n,p,d,r)}{\cG(n,p)} = \kl{\A}{\B}
&\le \sum_{k=0}^{n-1} \biggl(\frac{\alpha^2}{p(1-p)}\cdot\frac{k}{r^4d}
+ \frac{136\alpha^2}{p^2(1-p)^2}\cdot\frac{k^2}{r^4d}\biggr)\\
&\le \frac{\alpha^2}{2p(1-p)}\cdot\frac{n^2}{r^4d}
+ \frac{68\alpha^2}{3p^2(1-p)^2}\cdot\frac{n^3}{r^4d}.
\end{split}
\end{equation*}
The asymptotic in Theorem~\ref{th:imp} directly follows.

\section{Detecting geometry using the signed triangle statistic} \label{se:pos}
In this section, we show that the geometric structure in $\cG(n,p,d,r)$
can be detected in certain parameter regimes of $d$ and $r$
using the signed triangle statistic
proposed in~\cite{bubeck2016testing},
thus proving Theorem~\ref{th:pos}.

Let $\A$ be the adjacency matrix of a sample graph $G$ with edge probability
$p$.
The signed triangle over the vertices $\{i,j,k\}$ is defined as
\begin{equation*}
\tau_{\{i,j,k\}}(G) \coloneqq (a_{i,j}-p)(a_{j,k}-p)(a_{k,i}-p).
\end{equation*}
Then, the signed triangle statistic is the sum of all possible
signed triangles in the graph:
\begin{equation*}
\tau(G)
\coloneqq \sum_{\{i,j,k\}\in\binom{[n]}{3}} \tau_{\{i,j,k\}}(G).
\end{equation*}

In $\cG(n,p)$, due to the independence of edges,
the calculation of the mean and the variance of the signed triangle statistic
is straightforward, which has been
done in \cite{bubeck2016testing}.
We state the results here as the following lemma.
\begin{lemma}
The signed triangle statistic in $\cG(n,p)$ satisfies
\begin{equation*}
\E[\tau(\cG(n,p))] = 0
\qquad\text{and}\qquad
\Var[\tau(\cG(n,p))] \le n^3.
\end{equation*}
\end{lemma}

The main goal of this section is to estimate the mean and the variance of
the signed triangle statistic in $\cG(n,p,d,r)$.

\subsection{Estimating the mean in the large variance regime}
We start our discussion with the estimate for the mean in the regime
when $r$ is large, specifically $r/\log^2 r \gg d^{1/6}$.
Before diving into the details, we first present a concentration result which
provides a tail bound for the remainder of a linear approximation of $\sigma$.
\begin{lemma} \label{lm:linear_cct}
Suppose $\x, \y \in \RR^d$ are independent standard normal random vectors.
Denote $\eta \coloneqq \E[\sigma'(\inner{\x}{\y})]$ and define the remainder of
a linear approximation of $\sigma$ as
\begin{equation*}
g(x) \coloneqq \sigma(x) - p - \eta x.
\end{equation*}
Then, the tails of $g(\inner{\x}{\y})$ satisfy that for $t \ge 6$,
\begin{equation*}
\P\biggl(\abs{g(\inner{\x}{\y})} \ge \frac{3\alpha t}{2r^2} \biggr)
\le \exp\biggl(-\sqrt{\frac{t}{2e}}\biggr).
\end{equation*}
\end{lemma}
\begin{proof}
Let $\x',\y'\sim\cN(\0,\I_d)$ be independent copies of $\x, \y$.
For ease of notation, denote $Z \coloneqq \inner{\x}{\y}$
and $Z' \coloneqq \inner{\x'}{\y'}$.
Then, by definition,
\begin{equation*}
p = \E[\sigma(Z')]
\quad\text{and}\quad
\eta = \E[\sigma'(Z')].
\end{equation*}
Plugging them into $g(x)$ and by the triangle inequality, we have that
\begin{equation*}
\begin{split}
\abs{g(Z)}
&= \abs{\sigma(Z) - \E[\sigma(Z')] - \E[\sigma'(Z')] Z}
= \abs{\E_{Z'}[\sigma(Z)-\sigma(Z')-\sigma'(Z')(Z-Z')]
-\E[\sigma'(Z')Z']}\\
&\le \underbrace{\abs{\E_{Z'}[\sigma(Z)-\sigma(Z')-\sigma'(Z')(Z-Z')]}}_{V_1}
+\underbrace{\abs{\E[\sigma'(Z')Z']}}_{V_2}.
\end{split}
\end{equation*}
We first bound $V_1$ and $V_2$ from above and then utilize Gaussian
hypercontractivity.

By Taylor's theorem,
\begin{equation*}
\sigma(Z)-\sigma(Z')-\sigma'(Z')(Z-Z') = \frac{\sigma''(\xi)}{2}(Z-Z')^2
\end{equation*}
for some $\xi$ between $Z'$ and $Z$.
Then, we have that
\begin{equation*}
\begin{split}
V_1 &\le \E_{Z'}[\abs{\sigma(Z)-\sigma(Z')-\sigma'(Z')(Z-Z')}]
= \E_{Z'}\biggl[\frac{\abs{\sigma''(\xi)}}{2}(Z-Z')^2\biggr]
= \E_{Z'}\biggl[\frac{\abs{f'(\xi)}}{2r^2d}(Z-Z')^2\biggr]\\
&\le \frac{\alpha}{2r^2d}\E_{Z'}[(Z-Z')^2],
\end{split}
\end{equation*}
where we used the assumption \eqref{eq:A1} in the last inequality.

Since $Z' = \inner{\x'}{\y'} = \sum_{i=1}^d x'_iy'_i$,
we have that $\E[Z'] = 0$ and $\E[Z'^2]=d$.
Therefore,
\begin{equation*}
\E_{Z'}[(Z-Z')^2] = \E_{Z'}[Z^2-2ZZ'+Z'^2] = Z^2 + d.
\end{equation*}
Hence, we obtain that
\begin{equation*}
V_1 \le \frac{\alpha}{2r^2d}(Z^2 + d)
= \frac{\alpha}{2r^{2}d}Z^{2} + \frac{\alpha}{2r^{2}}.
\end{equation*}

Turning to $V_{2}$, by the triangle inequality we have that
\begin{equation*}
V_2 = \abs{\E[\sigma'(\inner{\x}{\y})\inner{\x}{\y}]}
= \biggl\lvert\sum_{i=1}^d \E[\sigma'(\inner{\x}{\y})x_iy_i]\biggr\rvert
\le \sum_{i=1}^d \abs{\E[\sigma'(\inner{\x}{\y})x_iy_i]},
\end{equation*}
Thus, by Stein's lemma and Jensen's inequality, we have that
\begin{equation*}
V_2 \le \sum_{i=1}^d \biggl\lvert
\E\biggl[y_i\frac{\partial \sigma'(\inner{\x}{\y})}{\partial x_i}\biggr]
\biggr\rvert
= \sum_{i=1}^d \abs{\E[y_i^2\sigma''(\inner{\x}{\y})]}
\le \sum_{i=1}^d \E[y_i^2\abs{\sigma''(\inner{\x}{\y})}]
\le \frac{\alpha}{r^2d}\E[\norm{\y}^2] = \frac{\alpha}{r^2}.
\end{equation*}

Therefore, by the triangle inequality, we have for any $q \geq 1$ that
\begin{equation*}
\norm{g(Z)}_q \le \norm{T_1}_q + \norm{T_2}_q
\le \frac{\alpha}{2r^2d}\norm{Z^2}_q + \frac{3\alpha}{2r^2}.
\end{equation*}

Since $Z^2 = (\sum_{i=1}^d x_i y_i)^2$ is a fourth order polynomial of
independent standard normal random variables,
by Gaussian hypercontractivity
(see \cite[Corollary~5.21]{boucheron2013concentration} for the univariate case
and \cite[Theorem~6.7]{janson1997gaussian} for a general argument),
we have that
\begin{equation*}
\norm{Z^2}_q \le (q-1)^2 \norm{Z^2}_2.
\end{equation*}
Since
\begin{equation*}
\E[Z^4] = \E\biggl[\biggl(\sum_{i=1}^d x_i y_i\biggr)^4\biggr]
= 3\sum_{i \ne j}\E[x_i^2] \E[x_j^2]\E[y_i^2]\E[y_j^2]\
+ \sum_{i}\E[x_i^4]\E[y_i^4]
= 3d^2 + 6d \le 9d^2,
\end{equation*}
we obtain that $\|Z^{2}\|_{2} \leq 3d$ and so
\begin{equation*}
\norm{g(Z)}_q \le \frac{3\alpha}{2r^2}(q^2-2q+2).
\end{equation*}

By Markov's inequality,
\begin{equation*}
\P\biggl(\abs{g(Z)} \ge \frac{3\alpha}{2r^2} t\biggr)
\le \biggl(\frac{3\alpha}{2r^2} t\biggr)^{-q}\E[\abs{g(Z)}^q]
\le t^{-q}(q^2-2q+2)^q
= \exp(-q \log t + q \log(q^2-2q+2)).
\end{equation*}
For $t \ge 3$, by choosing $q = \sqrt{t/e-1}+1$, we have that
\begin{equation*}
\P\biggl(\abs{g(Z)} \ge \frac{3\alpha t}{2r^2}\biggr)
\le \exp\biggl(-\biggl(\sqrt{\frac{t}{e}-1}+1\biggr)\biggr).
\end{equation*}
Hence, for $t \ge 6$,
\begin{equation*}
\P\biggl(\abs{g(Z)} \ge \frac{3\alpha t}{2r^2}\biggr)
\le \exp\biggl(-\sqrt{\frac{t}{2e}}\biggr). \qedhere
\end{equation*}
\end{proof}

Let
\begin{equation} \label{eq:lambda}
\lambda \coloneqq \E\biggl[
f\biggl(\frac{\inner{\x}{\y}-\mu_{p,d,r}}{r\sqrt{d}}\biggr)\biggr].
\end{equation}
Then,
\begin{equation}
\E[\sigma'(\inner{\x}{\y})] = \frac{\lambda}{r \sqrt{d}}.
\end{equation}
We also have
\begin{equation*}
0 \le \lambda \le \sup_x \abs{f(x)} \le 2\sqrt{\alpha}.
\end{equation*}
In the following lemma,
we show that $\lambda$ is bounded away from $0$ uniformly for all $r$ and $d$.
\begin{lemma} \label{lm:lambda_bounded}
Let $\lambda$ be defined in \eqref{eq:lambda}.
For a fixed $p \in (0, 1)$,
there exists a constant $C_p > 0$ that does not depend on $r$ and $d$
such that $\lambda \ge C_p$.
\end{lemma}
\begin{proof}
Let
\begin{equation*}
Z \coloneqq \frac{\inner{\x}{\y} - \mu_{p,d,r}}{r\sqrt{d}}.
\end{equation*}
By Proposition~\ref{pp:inner_gauss_cct}, we have
\begin{equation*}
\P\biggl(Z - \E[Z] \ge \frac{t}{r\sqrt{d}}\biggr) \le
\exp\biggl(-\frac{1}{2}\min\biggl\lbrace
\frac{t^2}{2d}, \frac{t}{\sqrt{2}}
\biggr\rbrace\biggr),
\end{equation*}
which gives
\begin{equation*}
\P(Z - \E[Z] \ge t) \le
\exp\biggl(-\frac{1}{2}\min\biggl\lbrace
\frac{t^2 r^2}{2}, \frac{t r \sqrt{d}}{\sqrt{2}}
\biggr\rbrace\biggr).
\end{equation*}
Since $r \ge 1$ and $d \ge 1$, we have for $t \ge 1$,
\begin{equation} \label{eq:Z_tail}
\P(Z - \E[Z] \ge t)
\le \exp\biggl(-\frac{t}{4}\biggr).
\end{equation}
The lower tail directly follows from $Z-\E[Z]$ being symmetric about zero.

Without loss of generality, we may assume $p \in (0, 1/2]$.
(If $p \in (1/2, 1)$, we can consider the connection function $1 - F(x)$,
which shares the same properties with $F(x)$.)
Since $F$ is strictly monotonic when $F \in (0, 1)$,
the inverse $F^{-1}$ exists in $(0, 1)$.
Let $z_{p/2} \coloneqq F^{-1}(p/2)$.
We first show by contradiction that $\E[Z] > z_{p/2} - 4\log\frac{4}{p}$.
Suppose that it does not hold, that is, $\E[Z] \le z_{p/2} - 4\log\frac{4}{p}$.
Then, by the monotonicity of $F$,
\begin{equation*}
\begin{split}
p &= \E[F(Z)] = \E[F(Z)\ind{Z < z_{p/2}}] + \E[F(Z)\ind{Z \ge z_{p/2}}]
\le \frac{p}{2} + \P(Z \ge z_{p/2})\\
&\le \frac{p}{2} + \P\biggl(Z \ge \E[Z] + 4\log\frac{4}{p}\biggr),
\end{split}
\end{equation*}
where the last inequality holds since
$\{Z \ge z_{p/2}\} \subset \{Z \ge \E[Z] + 4\log\frac{4}{p}\}$.
Using \eqref{eq:Z_tail}, we have that
\begin{equation*}
p \le \frac{p}{2} + \frac{p}{4} = \frac{3}{4}p,
\end{equation*}
which is a contradiction since $p > 0$.
Hence, $\E[Z] > z_{p/2} - 4\log\frac{4}{p}$.

Similarly, we show that $\E[Z] < z_{7p/4} + 4\log 3$,
where $z_{7p/4} \coloneqq F^{-1}(7p/4)$.
Again suppose $\E[Z] \ge z_{7p/4} + 4\log 3$.
Then, by the monotonicity of $F$,
\begin{equation*}
\begin{split}
p &= \E[F(Z)] \ge \E[F(Z)\ind{Z \ge z_{7p/4}}]
\ge \frac{7p}{4} \P(Z \ge z_{7p/4})
\ge \frac{7p}{4} \P(Z \ge \E[Z] - 4\log 3),
\end{split}
\end{equation*}
where the last inequality holds since
$\{Z \ge z_{7p/4}\} \supset \{Z \ge \E[Z] - 4\log 3\}$.
Therefore, by the lower tail bound,
\begin{equation*}
p \ge \frac{7p}{4} (1 - \P(Z < \E[Z] - 4\log 3))
\ge \frac{7p}{4}\biggl(1 - \frac{1}{3}\biggr) = \frac{7}{6}p,
\end{equation*}
which is a contradiction since $p > 0$.
Hence, $\E[Z] < z_{7p/4} + 4\log 3$.

Consider the interval
$I \coloneqq [z_{p/2} - 4\log\frac{4}{p} - 4, z_{7p/4} + 4\log 3 + 4]$.
The interval is bounded since $F$ is a CDF.
Since $f$ is a continuous function and strictly positive in the closed interval
by assumption~\eqref{eq:A0},
using the extreme value theorem,
we have that
$\inf_{x \in I} f(x) \ge f(\xi) = \epsilon_p > 0$ for some $\xi \in I$.
Therefore,
\begin{equation*}
\begin{split}
\E[f(Z)] &\ge \E[f(Z) \ind{Z \in I}] \ge \epsilon_p \P(Z \in I)\\
&= \epsilon_p \biggl(1 - \P\biggl(Z < z_{p/2} - 4\log\frac{4}{p} - 4\biggr)
- \P(Z > z_{7p/4} + 4\log 3 + 4)\biggr)\\
&\ge \epsilon_p (1 - \P(Z \le \E[Z] - 4) - \P(Z \ge \E[Z]+ 4))
\ge \biggl(1- \frac{2}{e}\biggr)\epsilon_p.
\end{split}
\end{equation*}
The claim directly follows.
\end{proof}

For ease of notation, we use the shorthand
\begin{equation*}
\brsigma(t) \coloneqq \sigma(t) - p.
\end{equation*}
Define the following events:
\begin{align*}
A_1 &\coloneqq \biggl\lbrace\biggl\lvert\brsigma(\inner{\x}{\y})
-\frac{\lambda}{r\sqrt{d}}\inner{\x}{\y}\biggr\rvert
\le \frac{3\alpha t}{2r^2}\biggr\rbrace,\\
A_2 &\coloneqq \biggl\lbrace\biggl\lvert\brsigma(\inner{\y}{\z})
-\frac{\lambda}{r\sqrt{d}}\inner{\y}{\z}\biggr\rvert
\le \frac{3\alpha t}{2r^2}\biggr\rbrace,\\
A_3 &\coloneqq \biggl\lbrace\biggl\lvert\brsigma(\inner{\z}{\x})
-\frac{\lambda}{r\sqrt{d}}\inner{\z}{\x}\biggr\rvert
\le \frac{3\alpha t}{2r^2}\biggr\rbrace,
\end{align*}
and let $A \coloneqq A_1 \cap A_2 \cap A_3$.
Then, on the event $A$ we have that
\begin{equation*}
\biggl\lvert
\biggl(\brsigma(\inner{\x}{\y})-\frac{\lambda}{r\sqrt{d}}\inner{\x}{\y}\biggr)
\biggl(\brsigma(\inner{\y}{\z})-\frac{\lambda}{r\sqrt{d}}\inner{\y}{\z}\biggr)
\biggl(\brsigma(\inner{\z}{\x})-\frac{\lambda}{r\sqrt{d}}\inner{\z}{\x}\biggr)
\biggr\rvert
\le \frac{27\alpha^3t^3}{8r^6}.
\end{equation*}
Expanding the product, on the event $A$ we have that
\begin{equation*}
\begin{split}
\tau(\x,\y,\z)
&\coloneqq
\brsigma(\inner{\x}{\y})\brsigma(\inner{\y}{\z})\brsigma(\inner{\z}{\x})\\
&\ge
\frac{\lambda^3}{r^3d^{3/2}}\inner{\x}{\y}\inner{\y}{\z}\inner{\z}{\x}
- \frac{27\alpha^3t^3}{8r^6}\\
&\phantom{{}\ge{}}
+ \frac{\lambda}{r\sqrt{d}}\biggl(
\inner{\x}{\y}\brsigma(\inner{\y}{\z})\brsigma(\inner{\z}{\x})
- \frac{\lambda}{r\sqrt{d}}\inner{\x}{\y}\inner{\y}{\z}\brsigma(\inner{\z}{\x})
\biggr)\\
&\phantom{{}\ge{}}
+ \frac{\lambda}{r\sqrt{d}}\biggl(
\inner{\y}{\z}\brsigma(\inner{\x}{\y})\brsigma(\inner{\z}{\x})
- \frac{\lambda}{r\sqrt{d}}\inner{\x}{\y}\inner{\z}{\x}\brsigma(\inner{\x}{\y})
\biggr)\\
&\phantom{{}\ge{}}
+ \frac{\lambda}{r\sqrt{d}}\biggl(
\inner{\z}{\x}\brsigma(\inner{\x}{\y})\brsigma(\inner{\y}{\z})
- \frac{\lambda}{r\sqrt{d}}\inner{\z}{\x}\inner{\x}{\y}\brsigma(\inner{\y}{\z})
\biggr).
\end{split}
\end{equation*}
The last three terms are identically distributed.
Denote
\begin{equation*}
\xi(\x,\y,\z) \coloneqq
\inner{\x}{\y}\brsigma(\inner{\y}{\z})\brsigma(\inner{\z}{\x})
-\frac{\lambda}{r\sqrt{d}}\inner{\x}{\y}\inner{\y}{\z}\brsigma(\inner{\z}{\x}).
\end{equation*}
Then, by taking the expectation on the event $A$, we have that
\begin{equation*}
\begin{split}
\E[\tau(\x,\y,\z)\bb1_A]
&\ge
\frac{\lambda^3}{r^3d^{3/2}}\E[\inner{\x}{\y}\inner{\y}{\z}\inner{\z}{\x}\bb1_A]
- \frac{27\alpha^3t^3}{8r^6}\E[\bb1_A]
+ \frac{3\lambda}{r\sqrt{d}}\E[\xi(\x,\y,\z)\bb1_A]\\
&\ge \frac{\lambda^3}{r^3d^{3/2}}\E[\inner{\x}{\y}\inner{\y}{\z}\inner{\z}{\x}]
- \frac{\lambda^3}{r^3d^{3/2}}
\abs{\E[\inner{\x}{\y}\inner{\y}{\z}\inner{\z}{\x}\bb1_{A^C}]}
- \frac{27\alpha^3t^3}{8r^6}\\
&\phantom{{}\ge{}} - \frac{3\lambda}{r\sqrt{d}}\abs{\E[\xi(\x,\y,\z)]}
- \frac{3\lambda}{r\sqrt{d}}\abs{\E[\inner{\x}{\y}\brsigma(\inner{\y}{\z})
\brsigma(\inner{\z}{\x})\bb1_{A^C}]}\\
&\phantom{{}\ge{}}- \frac{3\lambda^2}{r^2d}
\abs{\E[\inner{\x}{\y}\inner{\y}{\z}\brsigma(\inner{\z}{\x})\bb1_{A^C}]}.
\end{split}
\end{equation*}
As we will see, the first term in the above display is equal to
$\lambda^3/(r^3\sqrt{d})$,
while all other terms vanish when $r$ is large.
The rest of this subsection is devoted to bounding them.

By independence of the random vectors and their coordinates, we have that
\begin{equation*}
\begin{aligned}
\E[\inner{\x}{\y}] &= \E\biggl[\sum_{i=1}^d x_i y_i\biggr]
= \sum_{i=1}^d \E[x_i]\E[y_i] = 0,\\
\E[\inner{\x}{\y}\inner{\y}{\z}] &= \E\biggl[\biggl(\sum_{i=1}^d x_i y_i\biggr)
\biggl(\sum_{j=1}^d y_j z_j\biggr)\biggr]
= \sum_{i,j=1}^d \E[x_iy_iy_jz_j]
= 0,\\
\E[\inner{\x}{\y}\inner{\y}{\z}\inner{\z}{\x}]
&= \E\biggl[\biggl(\sum_{i=1}^d x_i y_i\biggr)
\biggl(\sum_{j=1}^d y_j z_j\biggr)
\biggl(\sum_{k=1}^d z_k x_k\biggr)\biggr]
= \sum_{i=1}^d \E[x_i^2 y_i^2 z_i^2] = d.
\end{aligned}
\end{equation*}
By similar combinatorial calculations, we also have that
\begin{align*}
\E[(\inner{\x}{\y})^2]
&= \sum_{i} \E[x_i^2y_i^2] = d,\\
\E[(\inner{\x}{\y}\inner{\y}{\z})^2]
&= \sum_{i,j} \E[x_i^2y_i^2y_j^2z_j^2]
= \sum_{i,j} \E[y_i^2y_j^2]
= d^2 + 2d \le 3d^2,\\
\E[(\inner{\x}{\y}\inner{\y}{\z}\inner{\z}{\x})^2]
&= \sum_{i,j,k} \E[x_i^2x_k^2]\E[y_i^2y_j^2]\E[z_j^2z_k^2]
+ \sum_{i \ne j}\E[x_i^2]\E[x_j^2]\E[y_i^2]\E[y_j^2]\E[z_i^2]\E[z_j^2]\\
&= d^3 + 10d^2 + 16d
\le 27d^3.
\end{align*}
Then, by Cauchy--Schwarz inequality, we obtain that
\begin{align*}
\abs{\E[\inner{\x}{\y}\inner{\y}{\z}\inner{\z}{\x} \bb1_{A^C}]}
&\le \sqrt{\E[(\inner{\x}{\y}\inner{\y}{\z}\inner{\z}{\x})^2]
\E[\bb1_{A^C}]}
\le 3d^{3/2}\sqrt{3\P(A^C)},\\
\abs{\E[\inner{\x}{\y}\inner{\y}{\z}\brsigma(\inner{\z}{\x})\bb1_{A^C}]}
&\le \sqrt{\E[(\inner{\x}{\y}\inner{\y}{\z})^2]
\E[\brsigma(\inner{\z}{\x})^2\bb1_{A^C}]}
\le d\sqrt{3\P(A^C)},\\
\abs{\E[\inner{\x}{\y}\brsigma(\inner{\y}{\z})
\brsigma(\inner{\z}{\x})\bb1_{A^C}]}
&\le \sqrt{\E[(\inner{\x}{\y})^2]
\E[\brsigma(\inner{\y}{\z})^2\brsigma(\inner{\z}{\x})^2
\bb1_{A^C}]}
\le \sqrt{d \P(A^C)}.
\end{align*}

As the last step, we bound $\abs{\E[\xi(\x,\y,\z)]}$ from above.
By Stein's lemma, we have that
\begin{equation*}
\begin{split}
\E[\inner{\x}{\y}\brsigma(\inner{\y}{\z})\brsigma(\inner{\z}{\x})]
&= \sum_{i=1}^d \E[x_i y_i \brsigma(\inner{\y}{\z})\brsigma(\inner{\z}{\x})]
= \sum_{i=1}^d \E[\E_{x_i}[x_i\brsigma(\inner{\y}{\z})]
\E_{y_i}[y_i \brsigma(\inner{\z}{\x})]]\\
&= \sum_{i=1}^d \E[z_i^2 \sigma'(\inner{\y}{\z})\sigma'(\inner{\z}{\x})]
= \E[\norm{\z}^2 \sigma'(\inner{\y}{\z})\sigma'(\inner{\z}{\x})].
\end{split}
\end{equation*}
Additionally with independence, we have that
\begin{equation*}
\begin{split}
\E[\inner{\x}{\y}\inner{\y}{\z}\brsigma(\inner{\z}{\x})]
&= \sum_{i,j=1}^d \E[x_i y_i y_j z_j\brsigma(\inner{\z}{\x})]
= \sum_{i=1}^d \E[x_i y_i^2 z_i\brsigma(\inner{\z}{\x})]
= \sum_{i=1}^d \E[x_i z_i\brsigma(\inner{\z}{\x})]\\
&= \sum_{i=1}^d \E[z_i^2\sigma'(\inner{\z}{\x})]
= \E[\norm{\z}^2\sigma'(\inner{\z}{\x})].
\end{split}
\end{equation*}
Therefore, by Cauchy--Schwarz inequality,
\begin{equation*}
\begin{split}
&\biggl\lvert
\E[\inner{\x}{\y}\brsigma(\inner{\y}{\z})\brsigma(\inner{\z}{\x})]
- \frac{\lambda}{r\sqrt{d}}
\E[\inner{\x}{\y}\inner{\y}{\z}\brsigma(\inner{\z}{\x})]
\biggr\rvert\\
&\quad= \abs{\E[\norm{\z}^2\sigma'(\inner{\z}{\x})
(\sigma'(\inner{\y}{\z})-\E[\sigma'(\inner{\y}{\z})])]}
\le \sqrt{\E[\norm{\z}^4\sigma'(\inner{\z}{\x})^2]
\Var[\sigma'(\inner{\y}{\z})]}.
\end{split}
\end{equation*}
Since
\begin{equation*}
\E[\norm{\z}^4] = \E\biggl[\biggl(\sum_{i=1}^d z_i^2\biggr)^2\biggr]
= \sum_{i \ne j}\E[z_i^2]\E[z_j^2] + \sum_{i}\E[z_i^4]
= d^2 + 2d \le 3d^2,
\end{equation*}
we have that
\begin{equation*}
\E[\norm{\z}^4\sigma'(\inner{\z}{\x})^2]
\le \frac{4\alpha}{r^2d} \E[\norm{\z}^4]
\le \frac{12\alpha d}{r^2}.
\end{equation*}
Since
\begin{equation*}
\frac{\partial \sigma'(\inner{\y}{\z})}{\partial y_i}
= z_i \sigma''(\inner{\y}{\z}),
\end{equation*}
by Gaussian Poincar\'e inequality,
\begin{equation*}
\Var[\sigma'(\inner{\y}{\z})]
\le \E[\norm{\nabla \sigma'(\inner{\y}{\z})}^2]
= \E[(\norm{\y}^2+\norm{\z}^2)\sigma''(\inner{\y}{\z})^2]
\le \frac{2\alpha^2}{r^4d}.
\end{equation*}
Hence, we obtain that
\begin{equation*}
\abs{\E[\xi(\x,\y,\z)]}
= \biggl\lvert
\E[\inner{\x}{\y}\sigma(\inner{\y}{\z})\sigma(\inner{\z}{\x})]
- \frac{\lambda}{r\sqrt{d}}
\E[\inner{\x}{\y}\inner{\y}{\z}\sigma(\inner{\z}{\x})]
\biggr\rvert
\le \frac{2\sqrt{6}\alpha^{3/2}}{r^3}.
\end{equation*}

Putting all the estimates together, we have that
\begin{equation*}
\E[\tau(\x,\y,\z)\bb1_A]
\ge \frac{\lambda^3}{r^3\sqrt{d}}
- \frac{6\sqrt{6}\alpha^{3/2}\lambda}{r^4\sqrt{d}}
- \frac{27\alpha^3t^3}{8r^6}
- 3\biggl(\frac{\lambda}{r}+\frac{\lambda^2}{r^2}+\frac{\lambda^3}{r^3}\biggr)
\sqrt{3\P(A^C)}.
\end{equation*}

Since by Cauchy--Schwarz inequality,
\begin{equation*}
\abs{\E[\tau(\x,\y,\z)\bb1_{A^C}]} \le \sqrt{\E[\tau(\x,\y,\z)^2]\E[\bb1_{A^C}]}
\le \sqrt{\P(A^C)},
\end{equation*}
we have that for $r \ge 2\sqrt{\alpha} \ge \lambda$,
\begin{equation*}
\E[\tau(\x,\y,\z)] = \E[\tau(\x,\y,\z)\bb1_A] + \E[\tau(\x,\y,\z)\bb1_{A^C}]
\ge \frac{\lambda^3}{r^3\sqrt{d}}
- \frac{12\sqrt{6}\alpha^2}{r^4\sqrt{d}}
- \frac{27\alpha^3t^3}{8r^6} - 10\sqrt{3\P(A^C)}.
\end{equation*}

A union bound gives
\begin{equation*}
\P(A^C) \le \sum_{i=1}^3 \P(A_i^C) \le 3\exp\biggl(-\sqrt{\frac{t}{2e}}\biggr),
\end{equation*}
where the tail bounds come directly from Lemma~\ref{lm:linear_cct}.
By choosing $t = 512e\log^2 r$, we have that
\begin{equation*}
\P(A^C) \le \frac{3}{r^{16}}.
\end{equation*}
To conclude,
\begin{equation*}
\E[\tau(\x,\y,\z)]
\ge \frac{\lambda^3}{r^3\sqrt{d}}
- \frac{12\sqrt{6}\alpha^2}{r^4\sqrt{d}}
- \frac{27(256e\alpha)^3\log^6 r}{r^6} - \frac{30}{r^8}.
\end{equation*}
Therefore, by assuming $r/\log^2 r \gg d^{1/6}$,
there exists a constant $C_{\alpha} > 0$
such that for $r \ge C_{\alpha}$,
\begin{equation*}
\E[\tau(\x,\y,\z)] \ge \frac{\lambda^3}{2r^3\sqrt{d}}.
\end{equation*}
Combining the above display with Lemma~\ref{lm:lambda_bounded},
we obtain the following lemma.
\begin{lemma} \label{lm:tau_large_r}
For $r/\log^2 r \gg d^{1/6}$,
there exist constants $C_{\alpha}, C_p > 0$
such that for $r \ge C_{\alpha}$,
\begin{equation*}
\E[\tau(\cG(n,p,d,r))] \ge \frac{C_pn^3}{r^3\sqrt{d}}.
\end{equation*}
\end{lemma}

\subsection{Estimating the mean in the high dimension regime}
In this part, we focus on the case when $d/\log^2 d \gg r^{6}$,
which complements the parameter regime discussed in the previous subsection.

We start by bounding the probability of the following two events in
$\cG(n,p,d,r)$:
\begin{equation*}
\Lambda \coloneqq \{1 \sim 2, 1 \sim 3\}
\qquad\text{and}\qquad
\Delta \coloneqq \{1 \sim 2, 2 \sim 3, 3 \sim 1\}.
\end{equation*}

Since
\begin{equation*}
\P(\Lambda) = \E[a_{1,2}a_{1,3}]
= \E[\sigma(\inner{\x_1}{\x_2})\sigma(\inner{\x_1}{\x_3})],
\end{equation*}
we directly have the following lemma as a consequence of \eqref{eq:gamma_upper}.
\begin{lemma} \label{lm:prob_cherry}
Let $\Lambda$ be defined above.
Then,
\begin{equation*}
\P(\Lambda) \le p^2 + \frac{\alpha^2}{r^4d}.
\end{equation*}
\end{lemma}

The probability of a triangle in $\cG(n,p,d,r)$ is given by
\begin{equation*}
\P(\Delta) = \E[a_{1,2}a_{2,3}a_{3,1}]
= \E[\sigma(\inner{\x_1}{\x_2})\sigma(\inner{\x_2}{\x_3})
\sigma(\inner{\x_3}{\x_1})].
\end{equation*}
We have the following estimate for the lower bound of the probability of
a triangle.
\begin{lemma} \label{lm:prob_tri}
Suppose $d/\log^2 d \gg r^6$.
There exists a constant $C_{\alpha} > 0$ such that
for $d \ge C_{\alpha}$,
\begin{equation*}
\P(\Delta) \ge p^3 + \frac{\lambda^3}{4r^3\sqrt{d}}.
\end{equation*}
\end{lemma}

Before proceeding to the proof of Lemma~\ref{lm:prob_tri},
we first show several technical results regarding the normal distribution.

Consider $\x, \y, \z \in \RR^d$.
Suppose $\z \sim \cN(\0, \I_d)$.
Conditioned on $\x, \y$,
the inner products $\inner{\x}{\z}$ and $\inner{\y}{\z}$
jointly have a bivariate normal distribution with mean $\0$ and covariant matrix
\begin{equation*}
\bSigma = \begin{bmatrix}
\norm{\x}^2 & \inner{\x}{\y}\\
\inner{\x}{\y} & \norm{\y}^2
\end{bmatrix}.
\end{equation*}Let
\begin{equation*}
\rho \coloneqq \frac{\inner{\x}{\y}}{\norm{\x}\norm{\y}}
\end{equation*}
be the correlation.
Then, the joint density can be written as
\begin{equation}
g(a, b) = \frac{1}{2\pi\norm{\x}\norm{\y}\sqrt{1-\rho^2}}
\exp\biggl(-\frac{1}{2(1-\rho^2)}\biggl(\frac{a^2}{\norm{\x}^2}
+ \frac{b^2}{\norm{\y}^2} - 2\rho\frac{ab}{\norm{\x}\norm{\y}}\biggr)\biggr).
\end{equation}
By the elementary inequality $\exp(x) \ge 1 + x$,
\begin{equation*}
\begin{split}
g(a, b) &\ge \frac{1}{2\pi\norm{\x}\norm{\y}\sqrt{1-\rho^2}}
\exp\biggl(-\frac{1}{2(1-\rho^2)}\biggl(\frac{a^2}{\norm{\x}^2}
+ \frac{b^2}{\norm{\y}^2} \biggr)\biggr)\biggl(1 + \frac{\rho}{1-\rho^2}
\frac{ab}{\norm{\x}\norm{\y}}\biggr)\\
&= \frac{1}{\norm{\x}\norm{\y}\sqrt{1-\rho^2}}
\varphi\biggl(\frac{a}{\norm{\x}\sqrt{1-\rho^2}}\biggr)
\varphi\biggl(\frac{b}{\norm{\y}\sqrt{1-\rho^2}}\biggr)\\
&\phantom{{}={}} +\frac{\rho}{\norm{\x}^2\norm{\y}^2(1-\rho^2)^{3/2}}ab
\varphi\biggl(\frac{a}{\norm{\x}\sqrt{1-\rho^2}}\biggr)
\varphi\biggl(\frac{b}{\norm{\y}\sqrt{1-\rho^2}}\biggr),
\end{split}
\end{equation*}
where $\varphi$ is the probability density function of the standard normal
distribution.

Consider a differentiable function $\sigma$.
By a change of variables $a' = a/(\norm{\x}\sqrt{1-\rho^2})$,
\begin{equation*}
\begin{split}
\int \sigma(a)
\varphi\biggl(\frac{a}{\norm{\x}\sqrt{1-\rho^2}}\biggr)\,da
&= \norm{\x}\sqrt{1-\rho^2}
\int \sigma(a'\norm{\x}\sqrt{1-\rho^2})
\varphi(a')\,da'\\
&= \norm{\x}\sqrt{1-\rho^2}
\E_{a'\sim\cN(0,1)}[\sigma(a'\norm{\x}\sqrt{1-\rho^2})]\\
&= \norm{\x}\sqrt{1-\rho^2}
\E_{\z\sim\cN(\0,\I_d)}[\sigma(\inner{\x}{\z}\sqrt{1-\rho^2})].
\end{split}
\end{equation*}
Similarly,
\begin{equation*}
\begin{split}
\int \sigma(a) a \varphi\biggl(\frac{a}{\norm{\x}\sqrt{1-\rho^2}}\biggr)\,da
&= \norm{\x}^2(1-\rho^2)
\int \sigma(a'\norm{\x}\sqrt{1-\rho^2}) a' \varphi(a')\,da'\\
&= \norm{\x}^2(1-\rho^2)
\E_{a'\sim\cN(0,1)}[\sigma(a'\norm{\x}\sqrt{1-\rho^2}) a'].
\end{split}
\end{equation*}
By Stein's lemma,
\begin{equation*}
\begin{split}
\E_{a'\sim\cN(0,1)}[\sigma(a'\norm{\x}\sqrt{1-\rho^2}) a']
&= \E_{a'\sim\cN(0,1)}
\biggl[\frac{\partial}{\partial a'}\sigma(a'\norm{\x}\sqrt{1-\rho^2})
\biggr]\\
&= \norm{\x}\sqrt{1-\rho^2} \E_{a'\sim\cN(0,1)}[
\sigma'(a'\norm{\x}\sqrt{1-\rho^2})]\\
&= \norm{\x}\sqrt{1-\rho^2} \E_{\z\sim\cN(\0,\I_d)}[
\sigma'(\inner{\x}{\z}\sqrt{1-\rho^2})].
\end{split}
\end{equation*}

Putting them together, we have that
\begin{equation} \label{eq:E_z_lower}
\begin{split}
\E_{\z}[\sigma(\inner{\x}{\z})\sigma(\inner{\y}{\z})]
&\ge \sqrt{1-\rho^2} \E_{\z}[\sigma(\inner{\x}{\z}\sqrt{1-\rho^2})]
\E_{\z'}[\sigma(\inner{\y}{\z'}\sqrt{1-\rho^2})]\\
&\phantom{{}\ge{}}+ (1-\rho^2)^{3/2}\inner{\x}{\y}
\E_{\z}[\sigma'(\inner{\x}{\z}\sqrt{1-\rho^2})]
\E_{\z'}[\sigma'(\inner{\y}{\z'}\sqrt{1-\rho^2})]].
\end{split}
\end{equation}
Using \eqref{eq:E_z_lower}, we have that
\begin{equation*}
\begin{split}
\P(\Delta)
&=\E[\sigma(\inner{\x}{\y})\sigma(\inner{\x}{\z})\sigma(\inner{\y}{\z})]
= \E[\sigma(\inner{\x}{\y})
\E_{\z}[\sigma(\inner{\x}{\z})\sigma(\inner{\y}{\z})]]\\
&\ge \underbrace{\E[\sqrt{1-\rho^2}\sigma(\inner{\x}{\y})
\E_{\z}[\sigma(\inner{\x}{\z}\sqrt{1-\rho^2})]
\E_{\z'}[\sigma(\inner{\y}{\z'}\sqrt{1-\rho^2})]}_{T_1}\\
&\phantom{{}\ge{}}+ \underbrace{\E
[(1-\rho^2)^{3/2}\inner{\x}{\y}\sigma(\inner{\x}{\y})
\E_{\z}[\sigma'(\inner{\x}{\z}\sqrt{1-\rho^2})]
\E_{\z'}[\sigma'(\inner{\y}{\z'}\sqrt{1-\rho^2})]]}_{T_2}.
\end{split}
\end{equation*}

The rest of the section is devoted to bounding $T_1$ and $T_2$
from below.

By Lemma~\ref{lm:sphere_inner_cct}, we have that
\begin{equation} \label{eq:rho_tail}
\P\biggl(\rho^2 \ge \frac{t^2}{d}\biggr)
\le 2\exp\biggl(-\frac{t^2}{4}\biggr).
\end{equation}
The following lemma gives a lower bound of $T_1$.
\begin{lemma} \label{lm:T_1}
For a constant $C_\alpha > 0$, when $d \ge C_\alpha$,
\begin{equation*}
T_1 \ge p^3 - \frac{C_\alpha \log d}{d}.
\end{equation*}
\end{lemma}

\begin{proof}
Denote the event
\begin{equation}
A \coloneqq \biggl\lbrace \rho^2 \ge \frac{t^2}{d} \biggr\rbrace.
\end{equation}
Since $\sigma$ is nonnegative,
\begin{equation*}
\begin{split}
T_1 &\ge \E[\sqrt{1-\rho^2}\sigma(\inner{\x}{\y})
\E_{\z}[\sigma(\inner{\x}{\z}\sqrt{1-\rho^2})]
\E_{\z'}[\sigma(\inner{\y}{\z'}\sqrt{1-\rho^2})]
\bb1_A].
\end{split}
\end{equation*}
Since $\sup_x \abs{f(x)} \le 2\sqrt{\alpha}$ by \eqref{eq:f_bound},
$\sigma$ is $\frac{2\sqrt{\alpha}}{r\sqrt{d}}$-Lipschitz.
Therefore,
\begin{equation*}
\abs{\sigma(\inner{\x}{\z}\sqrt{1-\rho^2}) - \sigma(\inner{\x}{\z})}
\le \frac{2\sqrt{\alpha}}{r\sqrt{d}}(1-\sqrt{1-\rho^2})\abs{\inner{\x}{\z}}
\le \frac{2\sqrt{\alpha}}{r\sqrt{d}}\rho^2\abs{\inner{\x}{\z}},
\end{equation*}
where the last inequality is due to
\begin{equation*}
1-\sqrt{1-\rho^2} = \frac{\rho^2}{1+\sqrt{1-\rho^2}} \le \rho^2.
\end{equation*}
Conditioned on $\x$, the inner product $\inner{\x}{\z}$ is distributed as
$\cN(0, \norm{\x}^2)$.
Then, by the property of a folded normal distribution,
$\E_{\z}[\abs{\inner{\x}{\z}}] = \norm{\x}\sqrt{2/\pi} \le \norm{\x}$.
Taking the expectation with respect to $\z$, and by Jensen's inequality,
we have that
\begin{equation*}
\begin{split}
\abs{\E_{\z}[\sigma(\inner{\x}{\z}\sqrt{1-\rho^2})]
-\E_{\z}[\sigma(\inner{\x}{\z})]}
&\le \E_{\z}[
\abs{\sigma(\inner{\x}{\z}\sqrt{1-\rho^2}) - \sigma(\inner{\x}{\z})}]\\
&\le \frac{2\sqrt{\alpha}}{r\sqrt{d}}\rho^2\E_{\z}[\abs{\inner{\x}{\z}}]
\le \frac{2\rho^2\sqrt{\alpha}}{r\sqrt{d}}\norm{\x}.
\end{split}
\end{equation*}
The same bound holds if we replace $\x$ with $\y$.
Hence, by combining them, we have that
\begin{equation*}
\abs{(\E_{\z}[\sigma(\inner{\x}{\z}\sqrt{1-\rho^2})]
-\E_{\z}[\sigma(\inner{\x}{\z})])
(\E_{\z'}[\sigma(\inner{\y}{\z'}\sqrt{1-\rho^2})]
-\E_{\z'}[\sigma(\inner{\y}{\z'})])}
\le \frac{4\rho^4\alpha}{r^2d}\norm{\x}\norm{\y}.
\end{equation*}
Expanding the product and using the triangle inequality, we obtain that
\begin{equation*}
\begin{split}
&\abs{\E_{\z}[\sigma(\inner{\x}{\z}\sqrt{1-\rho^2})]
\E_{\z'}[\sigma(\inner{\y}{\z'}\sqrt{1-\rho^2})]
- \E_{\z}[\sigma(\inner{\x}{\z})])\E_{\z'}[\sigma(\inner{\y}{\z'})])}\\
&\qquad\le \frac{4\rho^4\alpha}{r^2d}\norm{\x}\norm{\y}
+ \E_{\z}[\sigma(\inner{\x}{\z})]
\abs{(\E_{\z'}[\sigma(\inner{\y}{\z'}\sqrt{1-\rho^2})]
-\E_{\z'}[\sigma(\inner{\y}{\z'})])}\\
&\qquad\phantom{{}\le{}}+ \E_{\z'}[\sigma(\inner{\y}{\z'})]
\abs{(\E_{\z}[\sigma(\inner{\x}{\z}\sqrt{1-\rho^2})]
-\E_{\z}[\sigma(\inner{\x}{\z})])}\\
&\qquad\le \frac{4\rho^4\alpha}{r^2d}\norm{\x}\norm{\y}
+ \frac{2\rho^2\sqrt{\alpha}}{r\sqrt{d}}(\norm{\x}+\norm{\y}).\\
\end{split}
\end{equation*}
Therefore, by putting them together, we have that
\begin{equation*}
\begin{split}
T_1
&\ge \E[\sqrt{1-\rho^2}\sigma(\inner{\x}{\y})
\E_{\z}[\sigma(\inner{\x}{\z}\sqrt{1-\rho^2})]
\E_{\z'}[\sigma(\inner{\y}{\z'}\sqrt{1-\rho^2})]\bb1_A]\\
&\ge \biggl(1-\frac{t^2}{d}\biggr)\E[\sigma(\inner{\x}{\y})
\E_{\z}[\sigma(\inner{\x}{\z})]\E_{\z'}[\sigma(\inner{\y}{\z'})]\bb1_A]\\
&\phantom{{}\ge{}}- \frac{4\alpha t^4}{r^2d^3}
\E[\norm{\x}\norm{\y}\sigma(\inner{\x}{\y})]
- \frac{2t^2\sqrt{\alpha}}{rd^{3/2}}(\E[\norm{\x}\sigma(\inner{\x}{\y})]
+\E[\norm{\y}\sigma(\inner{\x}{\y})])\\
&\ge \E[\sigma(\inner{\x}{\y})
\E_{\z}[\sigma(\inner{\x}{\z})]\E_{\z'}[\sigma(\inner{\y}{\z'})]\bb1_A]
- \frac{t^2}{d}\E[\sigma(\inner{\x}{\y})
\E_{\z}[\sigma(\inner{\x}{\z})]\E_{\z'}[\sigma(\inner{\y}{\z'})]]\\
&\phantom{{}\ge{}}- \frac{4\alpha t^4}{r^2d^3}
\E[\norm{\x}]\E[\norm{\y}]
- \frac{2t^2\sqrt{\alpha}}{rd^{3/2}}(\E[\norm{\x}]+\E[\norm{\y}])\\
&\ge \E[\sigma(\inner{\x}{\y})
\E_{\z}[\sigma(\inner{\x}{\z})]\E_{\z'}[\sigma(\inner{\y}{\z'})]]
- \P(A^C) - \frac{t^2}{d}
- \frac{4\alpha t^4}{r^2d^3}\E[\norm{\x}]\E[\norm{\y}]\\
&\phantom{{}\ge{}}
- \frac{2t^2\sqrt{\alpha}}{rd^{3/2}}(\E[\norm{\x}]+\E[\norm{\y}]).
\end{split}
\end{equation*}
By \eqref{eq:rho_tail} and
$\E[\norm{\x}] \le \sqrt{\E[\norm{\x}^2]} \le \sqrt{d}$,
\begin{equation*}
T_1 \ge \E[\sigma(\inner{\x}{\y})
\E_{\z}[\sigma(\inner{\x}{\z})]\E_{\z'}[\sigma(\inner{\y}{\z'})]]
-2\exp\biggl(-\frac{t^2}{4}\biggr) -\frac{t^2}{d}
-\frac{4t^2\sqrt{\alpha}}{r d}
-\frac{4\alpha t^4}{r^2 d^2}.
\end{equation*}

By Fubini's theorem, we can interchange the expectation over $\x$ and $\z$:
\begin{equation*}
\E[\sigma(\inner{\x}{\y})
\E_{\z}[\sigma(\inner{\x}{\z})]\E_{\z'}[\sigma(\inner{\y}{\z'})]]
= \E[\E_{\x}[\sigma(\inner{\x}{\y})\sigma(\inner{\x}{\z})]
\E_{\z'}[\sigma(\inner{\y}{\z'})]].
\end{equation*}

Recall the definitions of $\eta$ and $\xi$
in the proof of Lemma~\ref{lm:gamma_Var}.
By the triangle inequality, we have
\begin{equation*}
\begin{split}
&\abs{\E[\E_{\x}[\sigma(\inner{\x}{\y})\sigma(\inner{\x}{\z})]
\E_{\z'}[\sigma(\inner{\y}{\z'})]] - p^3}
= \abs{\E[\xi(\y,\z) \eta(\y)] - p^3}\\
&\quad= \abs{\E[(\xi(\y,\z) - \E[\xi(\y,\z)]) \eta(\y)]
+ p(\E[\xi(\y,\z)] - p^2)}\\
&\quad\le \abs{\E[(\xi(\y,\z) - \E[\xi(\y,\z)]) \eta(\y)]}
+ p\abs{\E[\xi(\y,\z)] - p^2}.
\end{split}
\end{equation*}
By Jensen's inequality, we have
\begin{equation*}
\begin{split}
\abs{\E[(\xi(\y,\z) - \E[\xi(\y,\z)]) \eta(\y)]}
&\le \E[\abs{(\xi(\y,\z) - \E[\xi(\y,\z)])}\eta(\y)]
\le \E[\abs{(\xi(\y,\z) - \E[\xi(\y,\z)])}]\\
&\le \sqrt{\Var[\xi(\y,\z)]},
\end{split}
\end{equation*}
where we also used $0 \le \eta(\y) \le 1$.

Similar to Lemma~\ref{lm:gamma_Var}, we also have
\begin{equation*}
\Var[\xi(\y,\z)] \le \frac{68\alpha^2}{r^4 d}.
\end{equation*}
By \eqref{eq:gamma_upper}, we have
\begin{equation*}
\abs{\E[\xi(\y,\z)] - p^2} \le \frac{\alpha^2}{r^4 d}.
\end{equation*}

Putting them together, we obtain that
\begin{equation*}
\E[\sigma(\inner{\x}{\y})
\E_{\z}[\sigma(\inner{\x}{\z})]\E_{\z'}[\sigma(\inner{\y}{\z'})]]
\ge p^3 - \frac{69\alpha^2}{r^4 d}.
\end{equation*}

Therefore,
\begin{equation*}
T_1 \ge p^3  - \frac{69\alpha^2}{r^4 d}
- 2\exp\biggl(-\frac{t^2}{4}\biggr) -\frac{t^2}{d}
- \frac{4t^2\sqrt{\alpha}}{r d}
- \frac{4\alpha t^4}{r^2 d^2}.
\end{equation*}
By taking $t = \sqrt{4\log d}$, we conclude that
\begin{equation*}
T_1 \ge p^3 - \frac{69\alpha^2}{r^4 d}
- \frac{2(2\log d+1)}{d}
- \frac{16\sqrt{\alpha}\log d}{rd}
- \frac{64\alpha \log^2 d}{r^2d^2}.
\end{equation*}
The claim directly follows.
\end{proof}

The following lemma gives a lower bound of $T_2$.
\begin{lemma} \label{lm:T_2}
There exists a constant $C_\alpha$ such that
when  $d \ge r^4$ and $d \ge C_\alpha$,
\begin{equation*}
T_2 \ge \frac{\lambda^3}{2r^3\sqrt{d}}.
\end{equation*}
\end{lemma}
Before proving Lemma~\ref{lm:T_2}, we show the following simple fact.
\begin{lemma} \label{lm:lip_expect}
Let $f$ be an $L$-Lipschitz function.
Then, the function $g: \RR^d \to \RR$ given by
\begin{equation}
g(\x) = \E_{\z \sim \cN(\0, \I_d)}[f(\inner{\x}{\z})]
\end{equation}
is $\sqrt{\frac{2}{\pi}}L$-Lipschitz, that is,
\begin{equation*}
\abs{g(\x) - g(\y)} \le \sqrt{\frac{2}{\pi}}L \norm{\x-\y}.
\end{equation*}
\end{lemma}
\begin{proof}
An application of Jensen's inequality yields
\begin{equation*}
\abs{g(\x) - g(\y)}
= \abs{\E_{\z}[f(\inner{\x}{\z}) - f(\inner{\y}{\z})]}
\le \E_{\z}[\abs{f(\inner{\x}{\z}) - f(\inner{\y}{\z})}]
\le L \E_{\z}[\abs{\inner{\x-\y}{\z}}].
\end{equation*}
Conditioned on $\x$ and $\y$,
$\inner{\x-\y}{\z}$ is distributed as $\cN(0, \norm{\x-\y}^2)$.
Therefore, by the mean of a folded normal distribution,
\begin{equation*}
\abs{g(\x) - g(\y)} \le \sqrt{\frac{2}{\pi}} L \norm{\x-\y}. \qedhere
\end{equation*}
\end{proof}

By Lemma~\ref{lm:lip_expect}, since $\sigma'$ is
$\frac{\alpha}{r\sqrt{d}}$-Lipschitz,
\begin{equation} \label{eq:Ez_sigma_der_lip}
\begin{split}
\abs{\E_{\z}[\sigma'(\inner{\x}{\z}\sqrt{1-\rho^2})]
-\E_{\z}[\sigma'(\inner{\x}{\z})]}
&\le \sqrt{\frac{2}{\pi}}\cdot\frac{\alpha}{r\sqrt{d}} \norm{\x}
\abs{1 - \sqrt{1-\rho^2}}\\
&= \sqrt{\frac{2}{\pi}}\cdot\frac{\alpha}{r\sqrt{d}} \norm{\x}
\frac{\rho^2}{1+\sqrt{1-\rho^2}}
\le \frac{\alpha\rho^2}{r\sqrt{d}} \norm{\x}.
\end{split}
\end{equation}

We have the following lemma which states the sub-Gaussian tails of
$\E_{\z}[\sigma'(\inner{\x}{\z})]$.
\begin{lemma} \label{lm:Ez_tail}
The tails of $\E_{\z}[\sigma'(\inner{\x}{\z})]$ satisfy
\begin{equation*}
\P\biggl(\biggl\lvert\E_{\z}[\sigma'(\inner{\x}{\z})]
- \frac{\lambda}{r\sqrt{d}}\biggr\rvert
\ge \frac{\alpha t}{r^2d}
\biggr) \le 2\exp\biggl(-\frac{t^2}{2}\biggr).
\end{equation*}
\end{lemma}
\begin{proof}

Recall the definition of $\lambda$ in \eqref{eq:lambda}:
\begin{equation*}
\lambda \coloneqq \E\biggl[
f\biggl(\frac{\inner{\x}{\y}-\mu_{p,d,r}}{r\sqrt{d}}\biggr)\biggr].
\end{equation*}
Consider the function
\begin{equation*}
h(\x) \coloneqq \E_{\z}[\sigma'(\inner{\x}{\z})].
\end{equation*}
We have $\E[h(\x)] = \lambda/(r\sqrt{d})$.

Taking the partial derivative of $h(\x)$ with respect to $x_i$, we have that
\begin{equation*}
\begin{split}
\frac{\partial h(\x)}{\partial x_i}
= \frac{1}{r^2d} \E_{\z}\biggl[z_i
f'\biggl(\frac{\inner{\x}{\z}-\mu_{p,d,r}}{r\sqrt{d}}\biggr)\biggr]
= \frac{x_i}{r^3d^{3/2}} \E_{\z}\biggl[
f''\biggl(\frac{\inner{\x}{\z}-\mu_{p,d,r}}{r\sqrt{d}}\biggr)\biggr].
\end{split}
\end{equation*}
Let
\begin{equation*}
Y \coloneqq \frac{\inner{\x}{\z}-\mu_{p,d,r}}{r\sqrt{d}}.
\end{equation*}
Then, if we fix $\x$,
$Y$ is distributed as $\cN(-\mu_{p,d,r}/(r\sqrt{d}), \norm{\x}^2/(r^2d))$.
Therefore, since we assume \eqref{eq:A2}, by Stein's lemma,
\begin{equation*}
\begin{split}
\biggl\lvert\E_{\z}\biggl[
f''\biggl(\frac{\inner{\x}{\z}-\mu_{p,d,r}}{r\sqrt{d}}\biggr)\biggr]
\biggr\rvert
&= \abs{\E_Y[f''(Y)]}
= \frac{1}{\Var[Y]}
\biggl\lvert\E\biggl[\biggl(Y+\frac{\mu_{p,d,r}}{r\sqrt{d}}\biggr)
f'(Y)\biggr]\biggr\rvert\\
&\le \frac{1}{\Var[Y]} \sqrt{\Var[Y]\E[f'(Y)^2]}
\le \frac{\alpha r\sqrt{d}}{\norm{\x}},
\end{split}
\end{equation*}
where the first inequality is by Cauchy--Schwarz.

Hence, we have that
\begin{equation*}
\norm{\nabla h(\x)}^2
= \sum_{i=1}^d \frac{x_i^2}{r^6d^{3}} \E_{\z}\biggl[
f''\biggl(\frac{\inner{\x}{\z}-\mu_{p,d,r}}{r\sqrt{d}}\biggr)\biggr]^2
\le \frac{\norm{\x}^2}{r^6d^{3}} \cdot \frac{\alpha^2r^2d}{\norm{\x}^2}
= \frac{\alpha^2}{r^4d^2}.
\end{equation*}
Therefore, by the classical Gaussian concentration inequality
(see, e.g., \cite[Theorem~5.5]{boucheron2013concentration}),
the claim directly follows.
\end{proof}

With Lemma~\ref{lm:lip_expect} and Lemma~\ref{lm:Ez_tail} in place,
we now turn to proving Lemma~\ref{lm:T_2}.
\begin{proof}[Proof of Lemma~\ref{lm:T_2}]
A standard $\chi^2$ concentration (see, e.g.,
\cite[Example~2.11]{wainwright2019high}) gives
\begin{equation*}
\P(\abs{\norm{\x}^2 - d} \ge 2t\sqrt{d}) \le 2\exp\biggl(-\frac{t^2}{8}\biggr).
\end{equation*}

Denote the events
\begin{align*}
A_0 &\coloneqq \biggl\lbrace \rho^2 \le \frac{t^2}{d} \biggr\rbrace\\
A_1 &\coloneqq \lbrace \abs{\norm{\x}^2 - d} \le 2t\sqrt{d} \rbrace,\\
A_2 &\coloneqq \lbrace \abs{\norm{\y}^2 - d} \le 2t\sqrt{d} \rbrace,\\
A_3 &\coloneqq \biggl\lbrace \biggl\lvert\E_{\z}[\sigma'(\inner{\x}{\z})]
- \frac{\lambda}{r\sqrt{d}}\biggr\rvert \le \frac{\alpha t}{r^2d}
\biggr\rbrace,
\end{align*}
and let $A \coloneqq A_0 \cap A_1 \cap A_2 \cap A_3$.
Then, on the event $A$,
by the triangle inequality and using \eqref{eq:Ez_sigma_der_lip},
\begin{equation*}
\begin{split}
&\biggl\lvert \E_{\z}[\sigma'(\inner{\x}{\z}\sqrt{1-\rho^2})]
- \frac{\lambda}{r\sqrt{d}}\biggr\rvert\\
&\qquad\le \biggl\lvert \E_{\z}[\sigma'(\inner{\x}{\z}\sqrt{1-\rho^2})]
- \E_{\z}[\sigma'(\inner{\x}{\z})]\biggr\rvert
+ \biggl\lvert \E_{\z}[\sigma'(\inner{\x}{\z})]
- \frac{\lambda}{r\sqrt{d}}\biggr\rvert\\
&\qquad\le \frac{\alpha\rho^2}{r\sqrt{d}} \norm{\x}
+ \frac{\alpha t}{r^2d}
\le \frac{\alpha t^2}{rd}
\sqrt{1+\frac{2t}{\sqrt{d}}} + \frac{\alpha t}{r^2d}
\le \frac{\alpha t^2}{rd}
\biggl(1+\frac{t}{\sqrt{d}}\biggr) + \frac{\alpha t}{r^2d}.
\end{split}
\end{equation*}
Hence for $1 \le t \le \sqrt{d}$,
\begin{equation*}
\biggl\lvert \E_{\z}[\sigma'(\inner{\x}{\z}\sqrt{1-\rho^2})]
- \frac{\lambda}{r\sqrt{d}}\biggr\rvert \le \frac{3\alpha t^2}{rd}.
\end{equation*}
Since the same bound also holds for $\y$, by combining them, we have
\begin{equation*}
\biggl\lvert \biggl(\E_{\z}[\sigma'(\inner{\x}{\z}\sqrt{1-\rho^2})]
- \frac{\lambda}{r\sqrt{d}}\biggr)
\biggl(\E_{\z'}[\sigma'(\inner{\y}{\z'}\sqrt{1-\rho^2})]
- \frac{\lambda}{r\sqrt{d}}\biggr)\biggr\rvert\\
\le \frac{9\alpha^2t^4}{r^2d^2}.
\end{equation*}
Expanding the product and by the triangle inequality, we get that
\begin{equation*}
\begin{split}
&\biggl\lvert \E_{\z}[\sigma'(\inner{\x}{\z}\sqrt{1-\rho^2})]
\E_{\z'}[\sigma'(\inner{\y}{\z'}\sqrt{1-\rho^2})]
- \frac{\lambda^2}{r^2d}\biggr\rvert\\
&\qquad\le \frac{9\alpha^2t^4}{r^2d^2}
+\frac{\lambda}{r\sqrt{d}}
\biggl\lvert \E_{\z}[\sigma'(\inner{\x}{\z}\sqrt{1-\rho^2})]
- \frac{\lambda}{r\sqrt{d}}\biggr\rvert
+\frac{\lambda}{r\sqrt{d}}
\biggl\lvert \E_{\z}[\sigma'(\inner{\x}{\z}\sqrt{1-\rho^2})]
- \frac{\lambda}{r\sqrt{d}}\biggr\rvert\\
&\qquad\le \frac{9\alpha^2t^4}{r^2d^2} +\frac{6\alpha\lambda t^2}{r^2d^{3/2}}.
\end{split}
\end{equation*}

For simplicity of notation, let
\begin{equation*}
 h(\x,\y) \coloneqq
(1-\rho^2)^{3/2}\E_{\z}[\sigma'(\inner{\x}{\z}\sqrt{1-\rho^2})]
\E_{\z'}[\sigma'(\inner{\y}{\z'}\sqrt{1-\rho^2})].
\end{equation*}
Then, on the event $A$, we have that
\begin{equation*}
\begin{split}
\biggl\lvert h(\x,\y) - \frac{\lambda^2}{r^2d}\biggr\rvert
&\le (1-\rho^2)^{3/2}\biggl\lvert
\E_{\z}[\sigma'(\inner{\x}{\z}\sqrt{1-\rho^2})]
\E_{\z'}[\sigma'(\inner{\y}{\z'}\sqrt{1-\rho^2})]
- \frac{\lambda^2}{r^2d}\biggr\rvert\\
&\phantom{{}\le{}}+ \frac{\lambda^2}{r^2d} \abs{(1-\rho^2)^{3/2} - 1}\\
&\le \frac{9\alpha^2t^4}{r^2d^2} +\frac{6\alpha\lambda t^2}{r^2d^{3/2}}
+ \frac{\lambda^2}{r^2d} \abs{(1-\rho^2)^{3/2} - 1}.
\end{split}
\end{equation*}
By the elementary inequality $(1-x)^a \ge 1-ax$ for $a \ge 1$,
\begin{equation*}
1 - (1-\rho^2)^{3/2} \le \frac{3\rho^2}{2}.
\end{equation*}
Therefore, additionally with $\lambda \le 2\sqrt{\alpha}$, on the event $A$,
we have that
\begin{equation*}
\biggl\lvert h(\x,\y) - \frac{\lambda^2}{r^2d}\biggr\rvert
\le \frac{9\alpha^2t^4}{r^2d^2} +\frac{6\alpha\lambda t^2}{r^2d^{3/2}}
+ \frac{3\lambda^2t^2}{2r^2d^2}
\le \frac{9\alpha^2t^4}{r^2d^2} +\frac{12\alpha^{3/2} t^2}{r^2d^{3/2}}
+ \frac{6\alpha t^2}{r^2d^2}.
\end{equation*}

By Jensen's inequality,
\begin{equation*}
\begin{split}
&\biggl\lvert \E[\inner{\x}{\y}\sigma(\inner{\x}{\y})h(\x,\y)\bb1_A]
- \frac{\lambda^2}{r^2d} \E[\inner{\x}{\y}\sigma(\inner{\x}{\y})\bb1_A]
\biggr\rvert\\
&\qquad\le \E\biggl[\biggl\lvert \inner{\x}{\y}\sigma(\inner{\x}{\y})
\biggl(h(\x,\y) - \frac{\lambda^2}{r^2d}\biggr) \biggr\rvert\bb1_A\biggr]
\le \E\biggl[\biggl\lvert h(\x,\y) - \frac{\lambda^2}{r^2d}\biggr\rvert
\cdot \abs{\inner{\x}{\y}\sigma(\inner{\x}{\y})}\bb1_A\biggr]\\
&\qquad\le \biggl(\frac{9\alpha^2t^4}{r^2d^2}
+\frac{12\alpha^{3/2} t^2}{r^2d^{3/2}}
+ \frac{6\alpha t^2}{r^2d^2}\biggr)
\E[\abs{\inner{\x}{\y}}].
\end{split}
\end{equation*}
By Jensen's inequality, we have
$\E[\abs{\inner{\x}{\y}}] \le \sqrt{\E[(\inner{\x}{\y})^2]} = \sqrt{d}$.
Hence, we obtain that
\begin{equation} \label{eq:E_xy_sig_h_A}
\biggl\lvert \E[\inner{\x}{\y}\sigma(\inner{\x}{\y})h(\x,\y)\bb1_A]
- \frac{\lambda^2}{r^2d} \E[\inner{\x}{\y}\sigma(\inner{\x}{\y})\bb1_A]
\biggr\rvert
\le \frac{9\alpha^2t^4}{r^2d^{3/2}} +\frac{12\alpha^{3/2} t^2}{r^2d}
+ \frac{6\alpha t^2}{r^2d^{3/2}}.
\end{equation}

Since conditioned on $\x$, $\inner{\x}{\y} \sim \cN(0, \norm{\x}^2)$,
by Stein's lemma,
\begin{equation*}
\begin{split}
\E_{\y}[\inner{\x}{\y}\sigma(\inner{\x}{\y})]
= \norm{\x}^2\E_{\y}[\sigma'(\inner{\x}{\y})].
\end{split}
\end{equation*}
Hence, we have that
\begin{equation*}
\begin{split}
\E[\inner{\x}{\y}\sigma(\inner{\x}{\y})]
&\ge \E[\norm{\x}^2\E_{\y}[\sigma'(\inner{\x}{\y})]\bb1_{A_1}]
\ge (d-2t\sqrt{d})\E[\sigma'(\inner{\x}{\y})\bb1_{A_1}]\\
&= (d-2t\sqrt{d})(\E[\sigma'(\inner{\x}{\y})]
- \E[\sigma'(\inner{\x}{\y})\bb1_{A_1^C}])\\
&\ge \frac{\lambda\sqrt{d}}{r} - \frac{2\lambda t}{r}
- \frac{2\sqrt{\alpha d}}{r}\P(A_1^C).
\end{split}
\end{equation*}
By Cauchy--Schwarz inequality,
\begin{equation*}
\abs{\E[\inner{\x}{\y}\sigma(\inner{\x}{\y})\bb1_{A^C}]}
\le \sqrt{\E[(\inner{\x}{\y})^2\sigma(\inner{\x}{\y})^2]\P(A^C)}
\le \sqrt{\E[(\inner{\x}{\y})^2]\P(A^C)}
= \sqrt{d\P(A^C)}.
\end{equation*}
Therefore, we have that
\begin{equation} \label{eq:E_xy_sig_A}
\begin{split}
\E[\inner{\x}{\y}\sigma(\inner{\x}{\y})\bb1_A]
&= \E[\inner{\x}{\y}\sigma(\inner{\x}{\y})]
- \E[\inner{\x}{\y}\sigma(\inner{\x}{\y})\bb1_{A^C}]\\
&\ge \E[\inner{\x}{\y}\sigma(\inner{\x}{\y})]
- \abs{\E[\inner{\x}{\y}\sigma(\inner{\x}{\y})\bb1_{A^C}]}\\
&\ge \frac{\lambda\sqrt{d}}{r} - \frac{2\lambda t}{r}
- \frac{2\sqrt{\alpha d}}{r}\P(A_1^C) - \sqrt{d\P(A^C)}.
\end{split}
\end{equation}

Similarly, by Cauchy--Schwarz inequality,
\begin{equation} \label{eq:E_xy_sig_h_AC}
\begin{split}
\abs{\E[\inner{\x}{\y}\sigma(\inner{\x}{\y})h(\x,\y)\bb1_{A^C}]}
&\le \sqrt{\E[(\inner{\x}{\y})^2\sigma(\inner{\x}{\y})^2h(\x,\y)^2]\P(A^C)}\\
&\le \frac{2\sqrt{\alpha}}{r\sqrt{d}}\sqrt{\E[(\inner{\x}{\y})^2]\P(A^C)}
= \frac{2\sqrt{\alpha}}{r}\sqrt{\P(A^C)}.
\end{split}
\end{equation}

By a union bound,
\begin{equation*}
\begin{split}
\P(A^C) &= \P((A_0 \cap A_1 \cap A_2 \cap A_3)^C)
\le \P(A_0^C)+\P(A_1^C)+\P(A_2^C)+\P(A_3^C)\\
&\le 2\biggl(\exp\biggl(-\frac{t^2}{4}\biggr)
+ \exp\biggl(-\frac{t^2}{8}\biggr)
+\exp\biggl(-\frac{t^2}{8}\biggr)+\exp\biggl(-\frac{t^2}{2}\biggr)\biggr).
\end{split}
\end{equation*}
Taking $t=2\sqrt{6\log d}$, we have that
\begin{equation*}
\P(A_1^C) \le \frac{2}{d^{3}}
\quad\text{and}\quad
\P(A^C) \le \frac{8}{d^{3}}.
\end{equation*}
Using the above display
and putting \eqref{eq:E_xy_sig_h_A}, \eqref{eq:E_xy_sig_A},
and \eqref{eq:E_xy_sig_h_AC} together,
we have that
\begin{equation*}
\begin{split}
T_2 &= \E[\inner{\x}{\y}\sigma(\inner{\x}{\y})h(\x,\y)]
\ge \E[\inner{\x}{\y}\sigma(\inner{\x}{\y})h(\x,\y)\bb1_A]
- \abs{\E[\inner{\x}{\y}\sigma(\inner{\x}{\y})h(\x,\y)\bb1_{A^C}]}\\
&\ge \frac{\lambda^3}{r^3\sqrt{d}}
- \frac{5184\alpha^2 \log^2d}{r^2d^{3/2}}
- \frac{288\alpha^{3/2} \log d}{r^2d}
- \frac{144\alpha \log d}{r^2d^{3/2}}\\
&\phantom{{}\ge{}}- \frac{32\alpha^{3/2}\sqrt{6\log d}}{r^3d}
- \frac{4\alpha^{3/2}}{r^3d^{7/2}}
- \frac{4\sqrt{2}\alpha}{r^2d^2}
- \frac{4\sqrt{2\alpha}}{rd^{3/2}},
\end{split}
\end{equation*}
where we also used $\lambda \le 2\sqrt{\alpha}$.

Therefore, when $d \ge r^4$, there exists a constant $C_{\alpha} > 0$
such that for $d \ge C_{\alpha}$,
\begin{equation*}
T_2 \ge \frac{\lambda^3}{2r^3\sqrt{d}}.
\end{equation*}
Lemma~\ref{lm:T_2} is hence proved.
\end{proof}

Combining Lemma~\ref{lm:T_1} and Lemma~\ref{lm:T_2},
we conclude that when $d/\log^2 d \gg r^{6}$
there exists a constant $C_\alpha > 0$ such that
for $d \ge C_\alpha$,
\begin{equation*}
\P(\Delta) \ge T_1 + T_2
\ge p^3 + \frac{\lambda^3}{2r^3\sqrt{d}} - \frac{C_\alpha \log d}{d}.
\end{equation*}
Lemma~\ref{lm:prob_tri} directly follows.

\begin{lemma}
Assume $d/\log^2 d \gg r^{6}$.
Then, there exist constants $C_p, C_{\alpha} \ge 0$,
for $d \ge C_{\alpha}$,
\begin{equation*}
\E[\tau_{\{1,2,3\}}] \ge \frac{C_p}{r^3\sqrt{d}}.
\end{equation*}
\end{lemma}
\begin{proof}
The expected signed triangle can be written as
\begin{equation*}
\begin{split}
\E[\tau_{\{1,2,3\}}] &= \E[(a_{1,2}-p)(a_{2,3}-p)(a_{3,1}-p)]\\
&= \E[a_{1,2}a_{2,3}a_{3,1}]
- p(\E[a_{1,2}a_{2,3}]+\E[a_{1,2}a_{3,1}]+\E[a_{2,3}a_{3,1}])\\
&\phantom{{}={}}+ p^2(\E[a_{1,2}]+\E[a_{2,3}]+\E[a_{3,1}]) - p^3\\
&= \P(\Delta) - 3p\P(\Lambda) + 2p^2.
\end{split}
\end{equation*}

By Lemma~\ref{lm:prob_cherry} and Lemma~\ref{lm:prob_tri},
we have that for $d \ge C_{\alpha}$,
\begin{equation*}
\E[\tau_{\{1,2,3\}}] \ge \frac{\lambda^3}{4r^3\sqrt{d}}
- \frac{3p\alpha^2}{r^4d}.
\end{equation*}
Additionally with Lemma~\ref{lm:lambda_bounded},
the claim follows directly.
\end{proof}

The expected signed triangle statistic in $\cG(n,p,d,r)$ satisfies
\begin{equation*}
\E[\tau(\cG(n,p,d,r))]
\coloneqq \E\biggl[\sum_{\{i,j,k\} \in \binom{[n]}{3}} \tau_{\{i,j,k\}}\biggr]
= \sum_{\{i,j,k\} \in \binom{[n]}{3}} \E[\tau_{\{i,j,k\}}]
= \binom{n}{3} \E[\tau_{\{1,2,3\}}].
\end{equation*}

Putting them together, we have the following lemma.
\begin{lemma} \label{lm:tau_large_d}
When $d/\log^2 d \gg r^6$,
there exist constants $C_p, C_{\alpha} > 0$
such that for $d \ge C_{\alpha}$,
\begin{equation*}
\E[\tau(\cG(n,p,d,r))] \ge \frac{C_p n^3}{r^3\sqrt{d}}.
\end{equation*}
\end{lemma}

\subsection{Estimating the variance}
The variance of the signed triangle statistic in $\cG(n,p,d,r)$ can be written
as
\begin{equation}
\begin{split}
\Var[\tau(\cG(n,p,d,r))]
&= \sum_{\{i,j,k\}, \{i',j',k'\} \subset [n]} V_{\{i,j,k\},\{i',j',k'\}}\\
&= \binom{n}{3} V_{\{1,2,3\},\{1,2,3\}}
+ \binom{n}{4}\binom{4}{2} V_{\{1,2,3\},\{1,2,4\}}\\
&\phantom{{}={}}+ \binom{n}{5}\binom{5}{1}\binom{4}{2} V_{\{1,2,3\},\{1,4,5\}}
+ \binom{n}{6}\binom{6}{3} V_{\{1,2,3\},\{4,5,6\}},
\end{split}
\end{equation}
where $V_{\{i,j,k\},\{i',j',k'\}}$ is the covariance of two signed
triangles defined by
\begin{equation*}
V_{\{i,j,k\},\{i',j',k'\}}
\coloneqq \E_{\cG(n,p,d,r)}[\tau_{\{i,j,k\}}\tau_{\{i',j',k'\}}]
- \E_{\cG(n,p,d,r)}[\tau_{\{i,j,k\}}]^2.
\end{equation*}

Since two triangles that do not share a vertex are independent,
\begin{equation*}
V_{\{1,2,3\},\{4,5,6\}}
= \E[\E[\tau_{\{1,2,3\}}\mid\x_1,\x_2,\x_3]
\E[\tau_{\{4,5,6\}}\mid\x_4,\x_5,\x_6]]
- \E[\tau_{\{1,2,3\}}]\E[\tau_{\{4,5,6\}}]
= 0.
\end{equation*}

For a signed triangle, $\E_{\cG(n,p,d,r)}[\tau_{\{1,2,3\}}^2] \le 1$.
Then, we have that
\begin{equation*}
V_{\{1,2,3\},\{1,2,3\}} \le \E_{\cG(n,p,d,r)}[\tau_{\{1,2,3\}}^2] \le 1.
\end{equation*}

Before proceeding to bounding the other two covariances,
we present the following lemma which directly follows from the results of
previous parts.

\begin{lemma} \label{lm:scd_mmt}
Let $\x, \y, \z \sim \cN(\0, \I_d)$ be independent standard normal random
vectors.
Then, there exists a constant $C_\alpha$ such that
\begin{equation*}
\E[\E_{\z}[(\sigma(\inner{\x}{\z})-p)(\sigma(\inner{\y}{\z})-p)]^2]
\le \frac{C_\alpha}{r^4d}.
\end{equation*}
\end{lemma}

\begin{proof}
From the definition in \eqref{eq:gamma_def},
\begin{equation*}
\E_{\z}[(\sigma(\inner{\x}{\z})-p)(\sigma(\inner{\y}{\z})-p)] = \gamma(\x,\y).
\end{equation*}
By Lemma~\ref{lm:gamma_E} and Lemma~\ref{lm:gamma_Var},
we have that
\begin{equation*}
\E[\gamma(\x,\y)^2] = \Var[\gamma(\x,\y)] + \E[\gamma(\x,\y)]^2
\le \frac{68\alpha^2}{r^4d} + \frac{\alpha^{4}}{r^{8} d^{2}}
\le \frac{68\alpha^2+\alpha^4}{r^4d},
\end{equation*}
where we used that $r, d \geq 1$.
\end{proof}

\begin{lemma} \label{lm:tau_123_124}
There exists a constant $C_\alpha > 0$ such that
\begin{equation*}
\E[\tau_{\{1,2,3\}}\tau_{\{1,2,4\}}]
\le \frac{C_\alpha}{r^4d}.
\end{equation*}
\end{lemma}
\begin{proof}
For simplicity of notation, denote
\begin{equation*}
\brsigma_{i,j} \coloneqq \sigma(\inner{\x_i}{\x_j})-p.
\end{equation*}
Then,
\begin{equation*}
\begin{split}
\E[\tau_{\{1,2,3\}}\tau_{\{1,2,4\}}]
&= \E[\brsigma_{1,2}^2\brsigma_{1,3}\brsigma_{2,3}
\brsigma_{1,4}\brsigma_{2,4}]
\le \E[\brsigma_{1,3}\brsigma_{2,3}
\brsigma_{1,4}\brsigma_{2,4}]\\
&= \E[\E_{\x_3}[\brsigma_{1,3}\brsigma_{2,3}]
\E_{\x_4}[\brsigma_{1,4}\brsigma_{2,4}]]
= \E[\E_{\x_3}[\brsigma_{1,3}\brsigma_{2,3}]^2].
\end{split}
\end{equation*}
The claim then follows directly from Lemma~\ref{lm:scd_mmt}.
\end{proof}

By Lemma~\ref{lm:tau_123_124} we thus have that
\begin{equation*}
V_{\{1,2,3\},\{1,2,4\}} \le \E[\tau_{\{1,2,3\}}\tau_{\{1,2,4\}}]
\le \frac{C_\alpha}{r^4d}.
\end{equation*}

\begin{lemma} \label{lm:tau_123_145}
There exists a constant $C_\alpha > 0$ such that
\begin{equation*}
\E[\tau_{\{1,2,3\}}\tau_{\{1,4,5\}}]
\le \frac{C_\alpha}{r^4d}.
\end{equation*}
\end{lemma}
\begin{proof}
By the definition of the signed triangle,
\begin{equation*}
\begin{split}
\E[\tau_{\{1,2,3\}}\tau_{\{1,4,5\}}]
&= \E[\brsigma_{1,2}\brsigma_{2,3}\brsigma_{3,1}
\brsigma_{1,4}\brsigma_{4,5}\brsigma_{5,1}]
= \E[\E_{\x_2,\x_3}[\brsigma_{1,2}\brsigma_{2,3}\brsigma_{3,1}]
\E_{\x_4,\x_5}[\brsigma_{1,4}\brsigma_{4,5}\brsigma_{5,1}]]\\
&= \E[\E_{\x_2,\x_3}[\brsigma_{1,2}\brsigma_{2,3}\brsigma_{3,1}]^2],
\end{split}
\end{equation*}
where the last equality holds since $(\x_2,\x_3)$ and $(\x_4,\x_5)$
are identically distributed.

By Jensen's inequality,
\begin{equation*}
\begin{split}
\E_{\x_2,\x_3}[\brsigma_{1,2}\brsigma_{2,3}\brsigma_{3,1}]^2
= \E_{\x_2}[\brsigma_{1,2}\E_{\x_3}[\brsigma_{2,3}\brsigma_{3,1}]]^2
&\le \E_{\x_2}[\brsigma_{1,2}^2\E_{\x_3}[\brsigma_{2,3}\brsigma_{3,1}]^2]
\le \E_{\x_2}[\E_{\x_3}[\brsigma_{2,3}\brsigma_{3,1}]^2].
\end{split}
\end{equation*}

Therefore, by Lemma~\ref{lm:scd_mmt},
\begin{equation*}
\E[\tau_{\{1,2,3\}}\tau_{\{1,4,5\}}]
\le \E[\E_{\x_3}[\brsigma_{2,3}\brsigma_{3,1}]^2]
\le \frac{C_\alpha}{r^4d}. \qedhere
\end{equation*}
\end{proof}

By Lemma~\ref{lm:tau_123_145} we thus have that
\begin{equation*}
V_{\{1,2,3\},\{1,4,5\}} \le \E[\tau_{\{1,2,3\}}\tau_{\{1,4,5\}}]
\le \frac{C_\alpha}{r^4d}.
\end{equation*}

Putting these bounds together, we have the following lemma.
\begin{lemma} \label{lm:Var_tau}
There exists a constant $C_\alpha > 0$ such that
\begin{equation*}
\Var[\tau(\cG(n,p,d,r))]
\le n^3 + \frac{C_\alpha n^5}{r^4d}.
\end{equation*}
\end{lemma}

\subsection{Concluding the proof}
Combining Lemma~\ref{lm:tau_large_r}, Lemma~\ref{lm:tau_large_d},
and Lemma~\ref{lm:Var_tau}, we have for $d, r \ge C_\alpha$ that
when $d/\log^2 d \gg r^{6}$ or $r/\log^2 r \gg d^{1/6}$,
\begin{equation*}
\abs{\E[\tau(\cG(n,p,d,r))] - \E[\tau(\cG(n,p))]}
\ge \frac{C_p n^3}{r^3\sqrt{d}}
\end{equation*}
and
\begin{equation*}
\max\{\Var[\tau(\cG(n,p,d,r)],\Var[\tau(\cG(n,p))]\}
\le n^3 + \frac{C_\alpha n^5}{r^4d}.
\end{equation*}
Therefore, by Chebyshev's inequality,
\begin{equation*}
\tv{\cG(n,p,d,r)}{\cG(n,p)}
\ge 1 - \biggl(\frac{C_{p}'r^6d}{n^3}+\frac{C_{p,\alpha}'r^2}{n}\biggr)
\end{equation*}
for some constants $C_{p}', C_{p,\alpha}' > 0$.
Notice that $r^6d/n^3 \to 0$ implies $r^2/n \to 0$.
Theorem~\ref{th:pos} is hence proved.

\section*{Acknowledgements}
We thank Ramon van Handel for insightful comments on several results
and Jiacheng Zhang for suggesting the proofs of Lemma~\ref{lm:der_bounded}
and Lemma~\ref{lm:lambda_bounded}.
The authors acknowledge the generous support from NSF grant DMS-1811724.

\bibliographystyle{plainnat}

\end{document}